\documentclass[10pt,a4paper]{article}

\usepackage{amsfonts, amsmath, amssymb, amsthm, enumerate}
\usepackage{psfrag,graphicx}
\newtheorem{proposition}{Proposition}[section]
\newtheorem{theorem}[proposition]{Theorem}

\newtheorem{lemma}[proposition]{Lemma}

\newtheorem{definition}[proposition]{Definition}

\newcommand\sqr[2]{{\vcenter{\vbox{\hrule height.#2pt
   \hbox{\vrule width.#2pt height#1pt \kern#1pt
      \vrule width.#2pt}
   \hrule height.#2pt}}}}

\newcommand{\ds}{\displaystyle}

\renewcommand{\epsilon}{\varepsilon}
\newcommand{\eps}{\epsilon}

\newcommand{\BV}{\mathbf{BV}}

\newcommand{\C}[1]{{{\mathbf{C}^\mathbf{#1}}}}

\newcommand{\Lloc}[1]{\mathbf{L}^{\mathbf{1}}_{\mathrm{loc}}}

\newcommand{\R}{{\mathbb{R}}}

\newcommand{\N}{\mathbb{N}}

\newcommand{\Id}{I\!d}

\newcommand{\RII}{R_{I\!I}}
\newcommand{\RLL}{R_{L\!L}}
\newcommand{\RLI}{R_{L\!I}}
\newcommand{\RIL}{R_{I\!L}}

\newcommand{\xa}{\mathpzc{a}}
\newcommand{\xb}{\mathpzc{b}}

\DeclareMathAlphabet{\mathpzc}{OT1}{pzc}{m}{it}

\newcommand{\sgn}{\mathrm{sgn}}

\newcommand{\tv}{\mathop{\rm TV}}

\makeatletter

\newlength{\captionwidth}
\long\def\@makecaption#1#2{%
   \vskip 10\p@
   \setbox\@tempboxa\hbox{#1: #2}%
   \ifdim \wd\@tempboxa > \captionwidth 
       \hbox to\hsize{\hfil
       \parbox[t]{\captionwidth}{
       #1: #2\par}
       \hfil}
     \else
       \hbox to\hsize{\hfil\box\@tempboxa\hfil}%
   \fi}

\makeatother

\title{Global existence of solutions for a multi-phase flow: \\
a bubble in a liquid tube and related cases}

\author{}

\author{Debora Amadori\footnote{Department of Engineering and Computer Science and Mathematics, University of L'Aquila,
Italy}
\and Paolo Baiti\footnote{Department of Mathematics and Computer Science, University of Udine, Italy}
\and Andrea Corli\footnote{Department of Mathematics and Computer Science,
University of Ferrara, Italy}
\and Edda Dal Santo\footnotemark[2]
}

\date{February 20, 2015}

\begin{document}

\maketitle

\par\vspace*{-.03\textheight}{\center
{\emph{Dedicated to Professor Tai-Ping Liu on the occasion of his 70th birthday}}\par}

\begin{abstract}
In this paper we study the problem of the global existence (in time) of weak, entropic solutions to a system of three hyperbolic conservation laws, in one space dimension, for large initial data. The system models the dynamics of phase transitions in an isothermal fluid; in Lagrangian coordinates, the phase interfaces are represented as stationary contact discontinuities. We focus on the persistence of solutions consisting in three bulk phases separated by two interfaces. Under some stability conditions on the phase configuration and by a suitable front tracking algorithm we show that, if the $\BV$-norm of the initial data is less than an explicit (large) threshold, then the Cauchy problem has global solutions.
\end{abstract}

\smallskip

\noindent\textit{2010~Mathematics Subject Classification:} 35L65,
35L60, 35L67, 76T99.

\smallskip



\section{Introduction}\label{sec:intro}
\setcounter{equation}{0}

This paper concludes a long analysis, begun in \cite{ABCD, ABCD2}, concerning the global existence in time of weak entropy solutions for a system of conservation laws modeling phase transitions in a fluid. More precisely, the focus of the analysis is on the persistence of solutions with \emph{stationary interfaces}.

The system under consideration consists of three equations, namely,
\begin{equation}\label{eq:system}
\left\{
\begin{array}{ll}
v_t - u_x &= 0\,,\\
u_t + p(v,\lambda)_x &= 0\,,
\\
\lambda_t &= 0\,,
\end{array}
\right.
\end{equation}
where the state variables $(v,u,\lambda)\in \Omega:=]0,+\infty[ \times \R \times [0,1] $ denote the specific volume, the velocity and the mass-density fraction of the vapor in the fluid, respectively. System \eqref{eq:system} is the conservative part of a more complex model first introduced in \cite{Fan}, one of whose novelties lies in the pressure $p$, which depends not only on $v$ but also on $\lambda$. More precisely, the pressure is prescribed by the law
\begin{equation}\label{eq:pressure}
p(v,\lambda)=\frac{a^2(\lambda)}{v},
\end{equation}
where the $\C{1}$ function $a$ is assumed to be strictly positive in $[0,1]$. For example, we can take $p=(1+\lambda)k/v$ for some positive constant $k$.

A related example of pressure law occurs in the theory of (non-isothermal) ionized gases. In that case, we have $p = (1+\alpha)RT/mv$, where $R$, $m$ and $T$ are the universal gas constant, the molecular mass and the temperature, respectively, while $\alpha$ is the ionization degree; if $\alpha=0$ the gas is not ionized. In that model, however, $\alpha$ is usually assigned as a given function of both $p$ and $T$ by Saha's law \cite{Fermi} while, in the case under consideration, the mass-density fraction $\lambda$ is understood as an independent variable.

System \eqref{eq:system} is strictly hyperbolic in $\Omega$; the eigenvalues $\pm a/v$ are genuinely nonlinear while $0$ is linearly degenerate. As a consequence, the field associated to this latter eigenvalue supports contact discontinuities, which are understood in the model as \emph{phase interfaces}.

The Cauchy problem for \eqref{eq:system} includes the initial data
\begin{equation}
\left(v_o(x), u_o(x), \lambda_o(x) \right),\qquad v_o(x)\ge \underline{v}>0,\ x\in\mathbb{R}.
\label{eq:init-data-o}
\end{equation}
The global existence (in time) of solutions to the initial-value problem for any strictly hyperbolic system of conservation laws whose eigenvalues do not change type is well-known and can be proved either with the Glimm scheme or with a front tracking algorithm, see \cite{Dafermos, Bressanbook}. The case of \emph{large} initial data is a challenging problem and can be tackled only for special systems; then, the issue is \emph{to find classes of initial data with (large) total variation for which global solutions to the Cauchy problem exist}. This topic was studied in \cite{Nishida68} and \cite{NishidaSmoller} in the case of the isothermal, respectively isentropic, $p$-system, see also \cite{DiPerna1, DiPerna2}; these results were extended in the seminal papers \cite{LiuIB,Liu} to the case of nonisentropic gas dynamics. We also refer to \cite{GST2007} for the extension of Nishida's result to the initial-value problem in Special Relativity.

As far as \eqref{eq:system}-\eqref{eq:init-data-o} is concerned, a positive answer was first given in \cite{amadori-corli-siam} and then in \cite{Asakura-Corli}. In particular, in the former paper an \emph{explicit} threshold of the $\BV$-norm of the initial data was provided in order to have the global existence of solutions. Moreover, both papers require a sort of balance of the $\BV$-norm of the initial data: the larger the variation of $(v_o,u_o)$, the smaller the variation of $\lambda_o$ and vice-versa. It would be interesting to prove whether or not, for general $\BV$-data, the solution exists globally in time. Motivated by the techniques introduced in \cite{amadori-corli-siam}, we considered in \cite{ABCD} the special case of initial data \eqref{eq:init-data-o} where
\[
\lambda_o(x) = \left\{
\begin{array}{ll}
\lambda_\ell & \hbox{ if }x<0 \,,\\
\lambda_r & \hbox{ if }x>0\,,\\
\end{array}
\right.
\]
which models the dynamics of a two-phases fluid. In order to deal with this particular framework, an original Riemann solver was proposed by which we proved the global existence of solutions for a wide class of large initial data. Such results improved by far those of \cite{amadori-corli-siam} when adapted to that setting. We briefly mention that the problem of \emph{smooth} perturbations of Riemann-type solutions has been studied by many authors, see for instance \cite{Li-Yu}.

The next step regarded the case of initial data with \emph{two} phase interfaces, namely,
\begin{equation}
\lambda_o(x) = \left\{
\begin{array}{ll}
\lambda_\ell & \hbox{ if }x<\xa \,,\\
\lambda_m & \hbox{ if }\xa<x<\xb\,,\\
\lambda_r & \hbox{ if }x>\xb\,,\\
\end{array}
\right.
\label{eq:init-data}
\end{equation}
where $\lambda_\ell,\lambda_m,\lambda_r$ are constant in $[0,1]$ and $\xa<\xb$ are the location of the waves; denote $a_\ell=a(\lambda_\ell)$, $a_m=a(\lambda_m)$, $a_r=a(\lambda_r)$. Clearly, this case is much more complicated than the previous one, because of the possible bouncing back and forward of the waves in the middle region $[\xa,\xb]$. In \cite{ABCD2} we answered in the positive to the above issue by assuming the condition $a_m < \min\{a_\ell,a_r\}$; if $v$ is fixed, then the pressure in the middle region is lower than the pressure in the outside regions. We also recovered the results of \cite{ABCD} by passing to the limit when either $a_m\to a_\ell$ or $a_m\to a_r$. To give a physical flavor to the problem, assume for a moment that the function $a(\lambda)$ is increasing with $\lambda$, as is the interesting case in modeling. Then the condition above also means that the mass-density fraction of vapor inside $[\xa,\xb]$ is less than outside: for brevity, we address to this case as the \emph{drop} case. We refer to \cite{amadori-corli-siam, ABCD, ABCD2} for more details on the model and references.

In this paper, we deal with the two remaining cases, namely,
\[
\hbox{ either } \quad a_m > \min\{a_\ell,a_r\}\quad \hbox{ or } \quad a_\ell< a_m < a_r,
\]
since the case $a_\ell> a_m > a_r$ can be deduced by the latter. Reminding of the previous physical interpretation, we shall loosely address to these cases as the \emph{bubble} case and the \emph{increasing-pressure} case, respectively.
As in \cite{ABCD2}, we introduce two special Riemann solvers. Both of them replace the standard non-physical waves \cite{Bressanbook} by waves defined through \emph{integral curves}; this allows us to attach these waves to the phase waves. Then, following again the lines of \cite{ABCD2}, we introduce an \lq\lq asymmetrical\rq\rq\ Glimm functional $F$, in the sense that its interaction potential $Q$ takes into account only certain shock waves approaching the phase waves while, on the contrary, all rarefaction waves approaching the phase waves are included. We point out that the definition of $F$ differs in each of the three cases mentioned above.

A key feature of both \cite{ABCD2} and the current paper is the possible occurrence of a \emph{stability condition}, depending on the case under consideration, to control interactions in the middle zone. More precisely, such a condition is needed both in the drop and in the increasing-pressure cases and imposes bounds to the strengths of the phase waves. However, in the bubble case we require no stability condition and, as a consequence, any phase wave is admitted.

The plan of the paper is the following. The main result is stated in Section~\ref{sec:main}. In Section~\ref{sec:prelim} we introduce the Riemann solvers, the composite waves and, at last, we define the functional $F$; we also recall some background facts from \cite{amadori-corli-siam, ABCD, ABCD2}. Sections~\ref{sec:interbubble} and~\ref{sec:interdstep} focus on the bubble and on the monotone-pressure cases, respectively; in particular, we prove there that the functional $F$ decreases at every interaction. The last Section~\ref{sec:convergence} deals with the two cases at the same time; by showing the convergence of the front tracking algorithm we conclude the proof of the main result.


\section{Main Results}\label{sec:main}
\setcounter{equation}{0}

In this section we state the main results of this paper, which concerns the existence of solutions to the Cauchy problem \eqref{eq:system}-\eqref{eq:init-data-o},\eqref{eq:init-data}. Under the notation defined in the Introduction, we set
\begin{equation}\label{def:etazeta}
  \eta = 2\, \frac{a_m-a_\ell}{a_m+a_\ell}\,,\qquad \zeta=2\, \frac{a_r-a_m}{a_r+a_m}\,.
\end{equation}
The quantities $\eta$ and $\zeta$ range over $]-2,2[\,$; they are the strengths of the two contact discontinuities (see \eqref{eq:2strength} below) that support the phase interfaces of the model, located at $x=\xa$ and $x=\xb$, respectively. We denote by $\mathcal{L},\mathcal{M},\mathcal{R}$ the three regions separated by $\eta$ and $\zeta$ in the $(x,t)$-plane, see Figure~\ref{fig:0}. 

As already mentioned in the Introduction, we focus on the \emph{bubble} case, which corresponds to consider $\eta>0$ and $\zeta<0$, and on the \emph{increasing-pressure} case, which corresponds to both $\eta>0$ and $\zeta>0$.


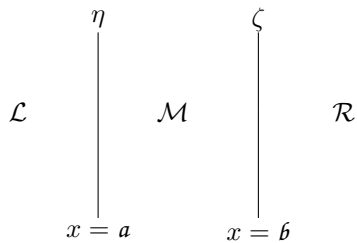
\begin{figure}[htbp]
\begin{picture}(100,100)(-80,-5)
\setlength{\unitlength}{1pt}

\put(170,0){
\put(0,0){\line(0,1){40}} \put(0,75){\makebox(0,0){$\zeta$}}
\put(-60,0){\line(0,1){40}} \put(-60,75){\makebox(0,0){$\eta$}}
\put(-60,-5){\makebox(0,0){$x=\xa$}}\put(0,-5){\makebox(0,0){$x=\xb$}}
\put(0,40){\line(0,1){30}} 
\put(-60,40){\line(0,1){30}} 
\put(-32,40){\makebox(0,0){$\mathcal{M}$}}
\put(32,40){\makebox(0,0){$\mathcal{R}$}}
\put(-90,40){\makebox(0,0){$\mathcal{L}$}}
}

\end{picture}

\caption{\label{fig:0}{The two phase waves with strengths $\eta$ and $\zeta$ in the $(x,t)$-plane and the regions $\mathcal{L},\mathcal{M},\mathcal{R}$.}}
\end{figure}


In order to state the existence theorem, we introduce some threshold functions. First, as in  \cite{ABCD,ABCD2}, we define the strictly decreasing function
\begin{equation}\label{eq:defK}
\mathcal{K}(r):=\frac{2}{1+r} \log\left(1+\frac{2}{r}\bigl(1+\sqrt{1+r}\bigr)\right)\,,\qquad r\in\R^+\,,
\end{equation}
which plays a key role in the main results; we notice that $\lim_{r\to 0^+}\mathcal{K}(r)=+\infty$ and $\lim_{r\to +\infty}\mathcal{K}(r)=0$. Moreover, as in \cite{ABCD2}, we need another function related to the stability of the two phase waves configuration, which differs from that in \cite{ABCD2}. In the bubble case it is
\begin{equation}\label{eq:b}
\mathcal{H}_b(|\eta|,|\zeta|):= \frac{4}{4-|\eta\zeta|}\max\left\{|\eta|\,\frac{2+|\zeta|}{2-|\zeta|},\,|\zeta|\,\frac{2+|\eta|}{2-|\eta|}\right\}\,,
\end{equation}
for $(|\eta|,|\zeta|)\in D_b:=[0,2[\times[0,2[\,$. As for the increasing-pressure case, the definition of the function $\mathcal{H}_c$ as well as that of its domain $D_c$ is more complicated though explicit; we refer to \eqref{eq:c} and \eqref{eq:dc} below, respectively, and to Figure~\ref{fig:Dc} for a picture of $D_c$. We can immediately observe that in the bubble case the pair $(|\eta|,|\zeta|)$ can vary inside the whole square $[0,2[\times[0,2[\,$, while in the increasing-pressure case they can cover only a portion of it.


\begin{figure}[htbp]
    \begin{center}
	\includegraphics[width=5cm]{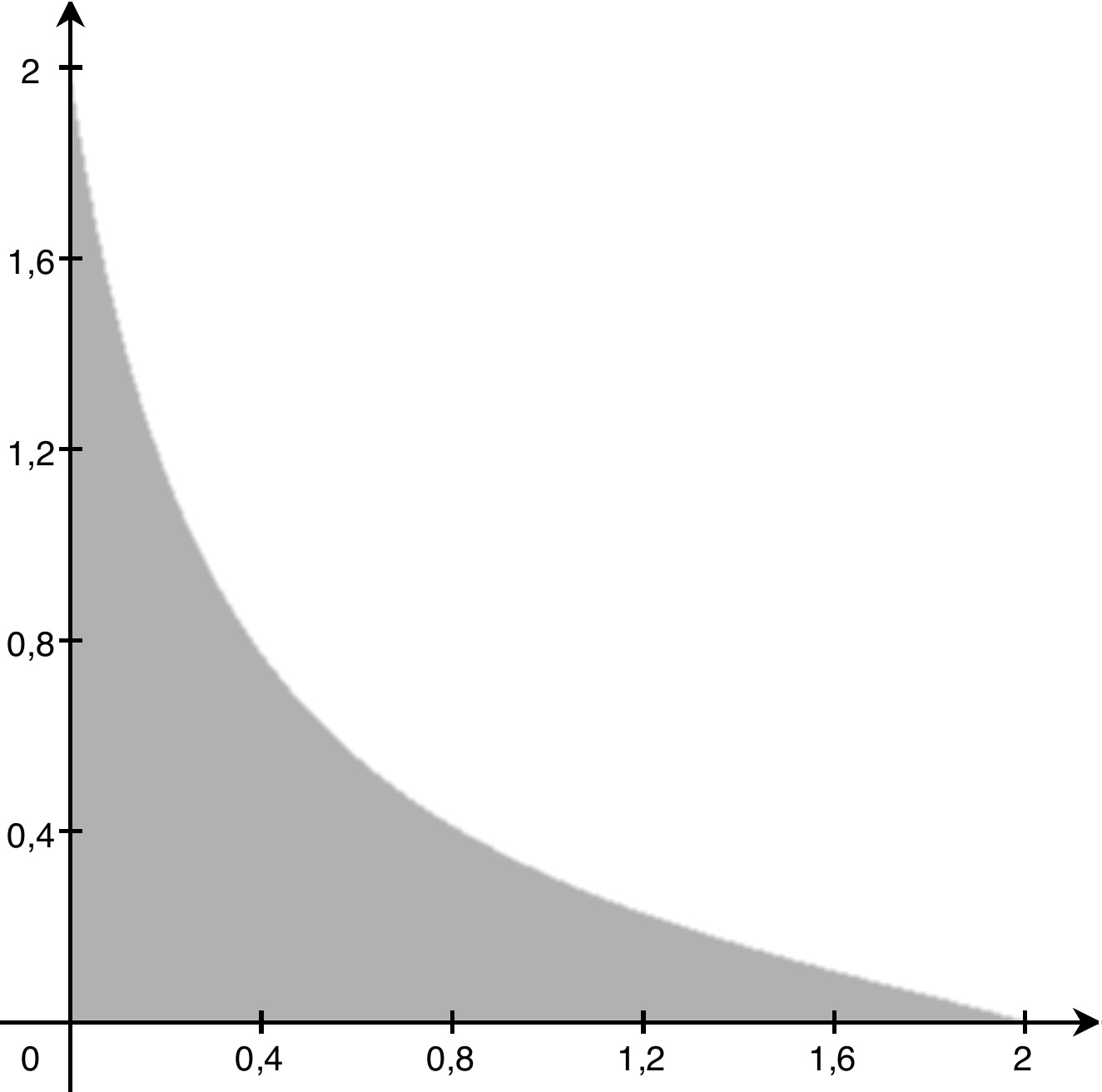}
   \begin{picture}(0,0)%
	    \put(-15,-5){$|\eta|$}%
	    \put(-155,130){$|\zeta|$}%
	\end{picture}%
    \end{center}
    \caption{\label{fig:Dc} The domain $D_c$ in the $(|\eta|,|\zeta|)$-plane.}
\end{figure}
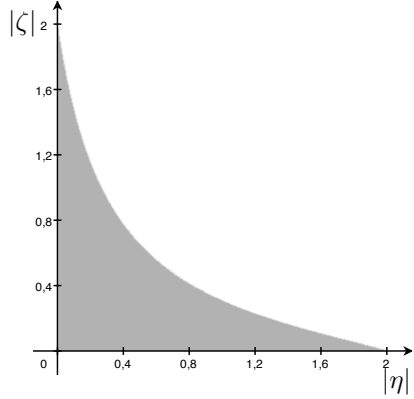


We denote $p_o(x) = p\left(v_o(x), \lambda_o(x) \right)$. 
The following theorem states the global in time existence of solutions in the bubble and in the increasing-pressure case. Notice that the statement is the same in both cases.

\begin{theorem}\label{thm:main}
Assume \eqref{eq:pressure} and consider initial data \eqref{eq:init-data-o},\eqref{eq:init-data}. Let $D=D_b$ and $\mathcal{H}=\mathcal{H}_b$ in the bubble case, $D=D_c$ and $\mathcal{H}=\mathcal{H}_c$ in the increasing-pressure case. Moreover, assume that for $\eta,\zeta$ as in \eqref{def:etazeta}, the pair $(|\eta|,|\zeta|)$ belongs to $D$. If
\begin{equation}\label{eq:hyp2}
\tv\left(\log(p_o)\right)+ \frac{1}{\min\{a_\ell,a_m,a_r\}}\tv\left(u_o\right) < \mathcal{K}\left(\mathcal{H}(|\eta|,|\zeta|)\right)
\end{equation}
holds, then the Cauchy problem \eqref{eq:system}-\eqref{eq:init-data-o},\eqref{eq:init-data} has a
weak entropic solution $(v,u,\lambda)$ defined for
$t\in\left[0,+\infty\right)$. If $\eta=\zeta=0$ the same conclusion holds with $\mathcal{K}\left(\mathcal{H}(|\eta|,|\zeta|)\right)$ replaced by $+\infty$ in \eqref{eq:hyp2}.

Moreover, the solution is valued in a compact set and $(v(t,\cdot),u(t,\cdot))\in L^\infty([0,\infty[; \BV(\R))$.
\end{theorem}

\smallskip

The sub-level sets $\mathcal{S}_h=\left\{(|\eta|,|\zeta|)\in\mathcal{D}:\   \mathcal{H}(|\eta|,|\zeta|) < h\right\}$, $h>0$, of the function $\mathcal{H}$ play an important role in condition \eqref{eq:hyp2}, see for instance Figure \ref{fig:HH} in the bubble case. Indeed, condition \eqref{eq:hyp2} holds for every $(|\eta|,|\zeta|)\in \mathcal{S}_h$ if
\begin{equation*}
\tv\left(\log(p_o)\right)+ \frac{1}{\min\{a_\ell,a_m,a_r\}}\tv\left(u_o\right)
< \mathcal{K}(h),
\end{equation*}
since $\mathcal{K}$ is decreasing. When $h = 2$, we have $\mathcal{K}(2)=2\log(2+\sqrt{3})/3$ and the case is particularly significative; let us consider, for example, the bubble framework. As in the drop case \cite{ABCD2}, the domain $\mathcal{S}_2$ includes the segments $[0, 2[$ on each axis, but the $2$-level set of $\mathcal{H}_b$ is no more the graph of the function $\zeta(|\eta|) =2(2-|\eta|)/(2 + |\eta|)$, see Figure~\ref{fig:HH}.


\begin{figure}[htbp]
   \begin{center}
       \includegraphics[width=8cm]{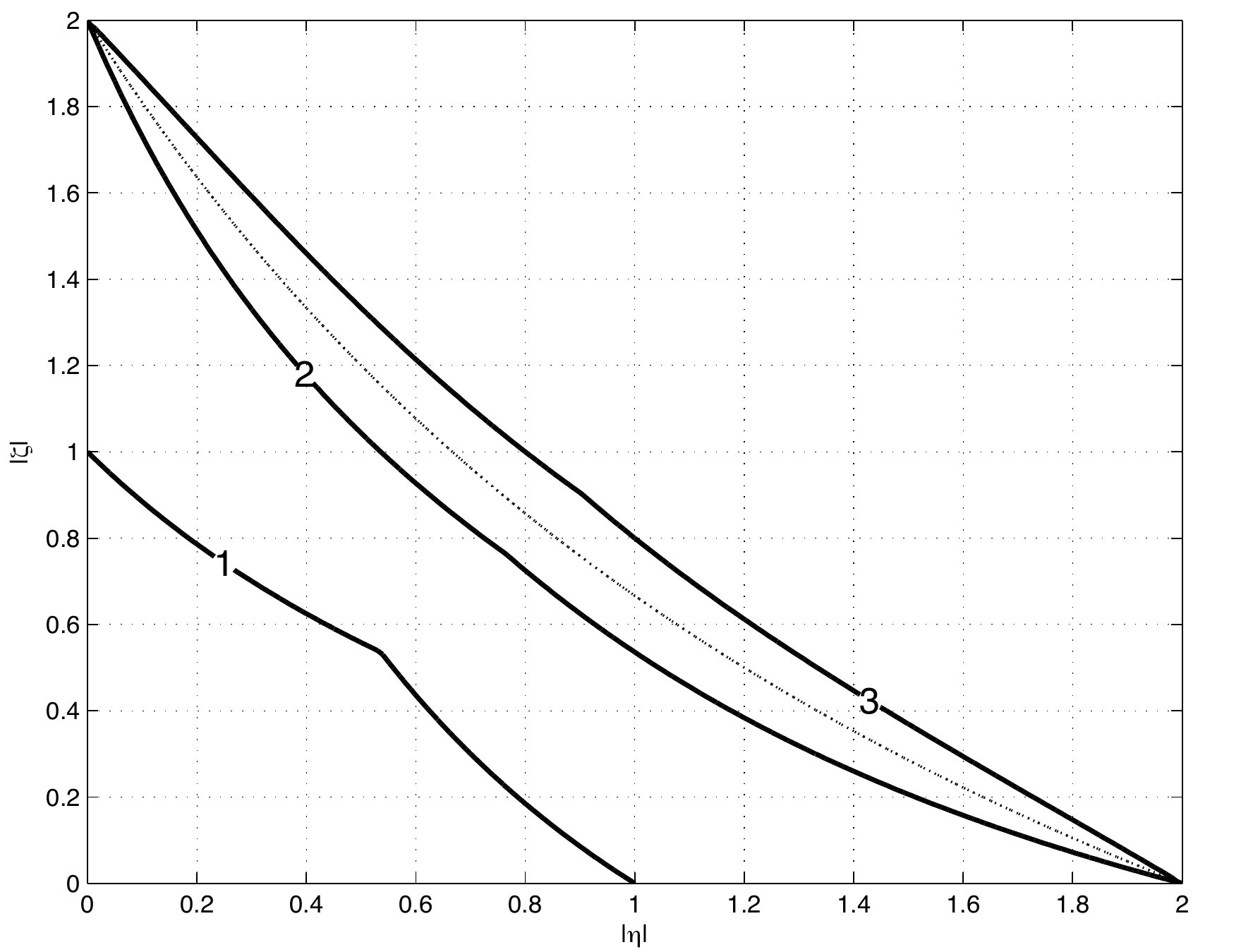}
   \end{center}
   \vspace{-3ex}
   \caption{\label{fig:HH}
   Sets of level $c$ of the function $\mathcal{H}_b$: cases $c=1,2,3$; the thin line is the curve $\zeta=\zeta(|\eta|)$.}
\end{figure}



\section{Preliminaries and Functionals}\label{sec:prelim}
\setcounter{equation}{0}

In this section we collect some preliminary results from \cite{amadori-corli-siam,ABCD,ABCD2}, focusing on the front tracking algorithm used to construct the approximate solutions. Moreover, we introduce some functionals needed to estimate the total variation of such solutions.

First, we recall some basic notions. System \eqref{eq:system} is strictly hyperbolic with two genuinely nonlinear characteristic fields (of family $1$ and $3$) and a linearly degenerate one of family $2$.
For $i=1,3$, the $i$-th right Lax curves through the point $\hat{U}=(\hat{v},\hat{u},\hat{\lambda})\in\Omega$ are
\begin{equation}\label{eq:lax13}
v \mapsto \left(v,\hat{u} + 2a(\hat{\lambda}) h(\eps_i),\hat{\lambda}\right)\,,\qquad v>0\,,
\end{equation}
where $\eps_i$ denotes the strength of an $i$-wave,
 \begin{gather}\label{eq:strengths}
\eps_1=\frac{1}{2}\log\left(\frac{v}{\hat{v}}\right),
\qquad \eps_3=\frac{1}{2}\log\left(\frac{\hat{v}}{v}\right),
\end{gather}
and $h$ is the function defined by
\begin{equation}\label{eq:h}
h(\eps)= \begin{cases}
\eps& \mbox{ if } \eps \ge 0\,,\\
\sinh \eps& \mbox{ if } \eps < 0\,.
\end{cases}
\end{equation}
Rarefaction waves have positive strength, while shock waves have negative strength. The wave curve for the second characteristic field through $\hat{U}\in\Omega$ is given by
\begin{equation*}
\lambda\mapsto\left(\hat{v}\ds\frac{a^2(\lambda)}{a^2(\hat{\lambda})},
\hat{u},\lambda\right)\,,\qquad \lambda\in[0,1],
\end{equation*}
and the strength of a $2$-wave is
\begin{equation}\label{eq:2strength}
\delta = 2\, \frac{a(\lambda)-a(\hat{\lambda})}{a(\lambda)+a(\hat{\lambda})}\,.
\end{equation}

%
%
We solve the Riemann problems by means of some \emph{Pre-Riemann solvers}, which are introduced in the following proposition. We use the symbols `$L$' and `$I$' to denote Lax and Integral curves, respectively.
For $i=1,3$ and $\theta_i\in\{L,I\}$, we define the functions
\begin{equation}\label{eq:choice}
\Theta_i=
\begin{cases}
h  &\text{if $\theta_i=L$}\,,\\
\Id &\text{if $\theta_i=I$}\,.
\end{cases}
\end{equation}

\begin{proposition}[Pre-Riemann solvers]\label{prop:preRsolver}
Fix $\theta_i\in\{L,I\}$, for $i=1,3$. There exists a map $R_{\theta_1\!\theta_3}:\Omega\times\Omega\to \R\times ]-2,2[\times \R$ such that for any two states $U_-=(v_-,u_-,\lambda_-),U_+=(v_+,u_+,\lambda_+)\in\Omega$ we have
\begin{equation}\label{eq:Rsolver}
R_{\theta_1\!\theta_3}(U_-,U_+)=
(\eps_1,\delta,\eps_3)\,,
\end{equation}
where $\eps_1,\delta,\eps_3$ represent waves of family $1,2,3$, respectively, satisfying the following relations:
\begin{equation}\label{impo}
\begin{gathered}
\eps_3-\eps_1=\frac{1}{2}\log\left(\frac{p_+}{p_-}\right)\,, \qquad a_-\Theta_1(\eps_1)+a_+\Theta_3(\eps_3)=\frac{u_+-u_-}{2}\,, \qquad
\delta=2\,\frac{a_+-a_-}{a_++a_-}\,.
\end{gathered}
\end{equation}
We denoted $a_{\pm}=a(\lambda_{\pm})$, $p_{\pm}=p(v_\pm,\lambda_{\pm})$.
\end{proposition}

We refer to \cite[Proposition 3.1]{ABCD2} for a proof of the previous result. In a few words, Proposition~\ref{prop:preRsolver} states the existence of four Pre-Riemann solvers $\RLL$, $\RII$, $\RLI$ and $\RIL$, that prescribe how to solve a Riemann problem with $i$-waves taken along (Lax or integral) $i$-curves, for $i=1,3$. In particular, $\RLL$ is the solver of \cite{amadori-corli-siam,ABCD} that employs Lax curves.

We use a front tracking algorithm \cite{Bressanbook} to build up the approximate solutions to \eqref{eq:system}-\eqref{eq:init-data-o},\eqref{eq:init-data}. The first step in the construction consists in taking a sequence $(v^\nu_o,u^\nu_o)_{\nu\in\N}$ of piecewise-constant functions with a finite number of jumps, that approximate the initial data \eqref{eq:init-data} in the sense of \cite[Section 4]{ABCD2}. We choose two parameters $\sigma=\sigma_\nu>0$, $\rho=\rho_\nu>0$ and proceed as follows. At time $t=0$ we apply the solver $\RLL$ at each jump of the approximated initial data; we split rarefactions into a finite number of rarefaction shocks, each of size $\le\sigma$, whose speed equals the characteristic speed at the right state (see \cite[Section 4]{ABCD2} for more details). Then, an approximate solution $(v^\nu,u^\nu,\lambda)(\cdot,t)$ is defined until the first time two wave fronts interact and a new Riemann problem arises. In the case of interaction between two fronts of families $1$ or $3$, we again use $\RLL$ and adopt the following strategy to approximate outgoing rarefaction waves: they are prolonged as a single discontinuity if they already existed before the interaction, otherwise they are split into a fan of waves as before. On the other hand, when solving an interaction of a wave of family $1$ or $3$ with a $2$-wave $\delta$, we possibly make use of a method that attaches certain reflected waves to $\delta$; the outcome is a wave of the same family of the incoming one and a stationary \emph{composite wave}, which is defined below.

\begin{definition}[Composite wave \cite{ABCD2}]\label{def:compositewave}
Consider two states $U_-=(v_-,u_-,\lambda_-)$ and $U_+=(v_+,u_+,\lambda_+)$ of $\Omega$, with $\lambda_-\neq\lambda_+$. The composite wave $\delta_0=(\delta_0^1,\delta,\delta_0^3)$ connecting $U_-$ to $U_+$ is the stationary wave defined by $\delta_0=\RII(U_-,U_+)$. We write $|\delta_0|=|\delta_0^1|+|\delta_0^3|$.
\end{definition}


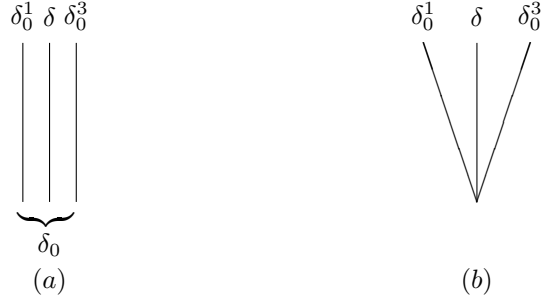
\begin{figure}[htbp]
\begin{picture}(100,110)(-120,10)
\setlength{\unitlength}{1pt}

\put(0,0){
\put(10,40){\line(0,1){60}}
\put(20,40){\line(0,1){60}}
\put(30,40){\line(0,1){60}}
\put(10,110){\makebox(0,0){$\delta_0^1$}}
\put(20,110){\makebox(0,0){$\delta$}}
\put(30,110){\makebox(0,0){$\delta_0^3$}}

\put(8,37){$\underbrace{\phantom{aaaa}}$}
\put(20,24){\makebox(0,0){$\delta_0$}}

\put(20,10){\makebox(0,0){$(a)$}}
}

\put(220,0){
\put(-40,40){\line(1,3){20}} \put(-40,40){\line(0,1){60}}
\put(-40,40){\line(-1,3){20}}
\put(-60,110){\makebox(0,0){$\delta_0^1$}}
\put(-40,110){\makebox(0,0){$\delta$}}
\put(-20,110){\makebox(0,0){$\delta_0^3$}}

\put(-40,10){\makebox(0,0){$(b)$}}
}

\end{picture}
\caption{\label{fig:1}{A composite wave $\delta_0$ in the $(x,t)$-plane: in $(a)$ it is drawn as three parallel close lines, while the auxiliary picture $(b)$ is used to determine the states in the interactions.}}
\end{figure}


\noindent Notice that in Definition~\ref{def:compositewave} the waves $\delta_0^1,\delta_0^3$ are given zero speed and $\delta_0$ reduces to a $2$-wave as long as $\delta_0^1=\delta_0^3=0$. Hence, instead of dealing with $2$-waves, we are left with composite waves belonging to a fictitious $0$-family.

When a wave front of family $1$ or $3$ with strength $\delta_i$ interacts with a composite wave $\delta_0=(\delta_0^1,\delta,\delta_0^3)$ at $t>0$, we exploit two different procedures to solve the emerging Riemann problem of states $U_-,U_+$. We indicate by $I_i(\eps)(U)$ the integral curves of family $i=1,3$, parametrized by $\eps$ and of origin $U\in \Omega$; then, we define $\widetilde{U}_-=I_1(\delta_0^1)(U_-)$ and
$\widetilde{U}_+=I_3(-\delta_0^3)(U_+)$. Here follows a description of the Riemann solvers; for a proof of the following facts see \cite[Proposition 3.7]{ABCD2}.
    \begin{enumerate}
\item \emph{Accurate Riemann solver}. If $|\delta_i|\ge \rho$, then the solution is formed by waves $\eps_1,\eps_0,\eps_3$, where $(\eps_1,\delta,\eps_3)=\RLL(\widetilde{U}_-,\widetilde{U}_+)$ and $\eps_0=\delta_0$.
\item \emph{Simplified Riemann solver}. If $|\delta_i|<\rho$, then we distinguish case $i=1$ and $i=3$:
\begin{itemize}
\item[i)] for $i=1$, the solution is formed by waves $\eps_1,\eps_0$ such that $(\eps_1,\delta,\eps_3)=\RLI(\widetilde{U}_-,\widetilde{U}_+)$ and $\eps_0=(\delta_0^1,\delta,\delta_0^3+\eps_3)$;
\item[ii)] for $i=3$, the solution is formed by waves $\eps_0,\eps_3$ such that $(\eps_1,\delta,\eps_3)=\RIL(\widetilde{U}_-,\widetilde{U}_+)$ and $\eps_0=(\delta_0^1+\eps_1,\delta,\delta_0^3)$.
\end{itemize}
\end{enumerate}
We emphasize that $\Theta_i=h$ in all cases and the following relations are verified:
\begin{gather}
\eps_3-\eps_1=\begin{cases}
-\delta_1 &\quad \text{if $i=1$},\\
\delta_3 &\quad \text{if $i=3$},
\end{cases}
\qquad a_- \Theta_1(\eps_1)+a_+ \Theta_3(\eps_3)=\begin{cases}
a_+ \Theta_1(\delta_1) &\quad \text{if $i=1$},\\
a_-\Theta_3(\delta_3) &\quad \text{if $i=3$},
\end{cases}\label{eq:rels}\\
\sgn\,\eps_i = \sgn\,\delta_i,
\qquad
\sgn\,\eps_j  = \begin{cases}
\sgn\,\delta\cdot \sgn\,\delta_i &\hbox{ if } i=1,
\\
-\sgn\,\delta\cdot \sgn\,\delta_i &\hbox{ if } i=3.
\end{cases}\label{eq:sign}
\end{gather}

\bigskip

We recall the interaction estimates \cite[Lemma 5.2 and 5.4]{ABCD2}. For $i=1,3$, we denote by $\eps_i$ the strength of the transmitted wave and by $\eps_j$, $j=1,3,\,j\neq i$, the strength of the reflected one (even in case of interaction with a composite wave treated by the Simplified solver).

\begin{lemma}[Interaction estimates]\label{lem:intest}
For the interaction between two waves at time $t>0$ we have the following; let $i,j=1,3$, $j\neq i$.
\begin{enumerate}
\item Assume that an $i$-wave $\delta_i$ interacts with a composite wave $\delta_0=(\delta_0^1,\delta,\delta_0^3)$. If $|\delta_i|\ge \rho$, we have
\begin{equation}\label{eq:intest1}
|\eps_i-\delta_i|=|\eps_j|\le
\frac{1}{2}|\delta_i\delta| \qquad \text{and} \qquad |\eps_0-\delta_0|=0,
\end{equation}
while, if $|\delta_i|<\rho$, it holds
\begin{equation}\label{eq:intest2}
|\eps_i-\delta_i|=|\eps_0-\delta_0|=|\eps_j|\le\begin{cases}
\ds \frac{C_o}{2}|\delta_i\delta| & \text{if $\delta_i<0$ and either ($i=1$, $\delta>0$) or ($i=3$, $\delta<0$)},\\[7pt]
\ds \frac{1}{2}|\delta_i\delta| & \text{otherwise},
\end{cases}
\end{equation}
where
\begin{equation}\label{eq:C_o}
C_o=C_o(\rho)=\frac{\sinh \rho}{\rho}\,.
\end{equation}
\item If two waves $\alpha_i$ and $\beta_j$ of different families interact with each other, then,
\begin{equation}\label{eq:intest3}
|\eps_i|=|\alpha_i|\,, \qquad |\eps_j|=|\beta_j|\,.
\end{equation}
\item Assume that two waves $\alpha_i$ and $\beta_i$ of the same family interact. If both $\alpha_i$ and $\beta_i$ are shocks, then the reflected wave $\eps_j$ is a rarefaction, while the transmitted wave $\eps_i$ is a shock and satisfies
\begin{equation}\label{eq:intest4}
|\eps_i|>\max \{|\alpha_i|,|\beta_i|\}\,.
\end{equation}
If $\alpha_i$ and $\beta_i$ have different signs, e.g. $\alpha_i<0<\beta_i$, then the reflected wave is a shock, both the amounts of shocks and rarefactions of the $i$-th family decrease across the interaction and one has
    \begin{align}\label{eq:chi_def}
    |\eps_j| & \le c(\alpha_i) \cdot \min\{|\alpha_i|,|\beta_i|\}\,,\qquad c(z) = \frac{\cosh z -1}{\cosh z+1}\,.
    \end{align}
\end{enumerate}
\end{lemma}
\noindent About the term in $C_o$ in \eqref{eq:C_o}, we notice that $C_o(\rho)>1$ for any $\rho>0$ and $C_o(\rho)\to 1^+$ for $\rho\to 0^+$.

\bigskip

Now, we introduce some functionals needed to prove the boundedness of the total variation of the approximate solutions. Using indices $\ell,m,r$ to refer to waves in $\mathcal{L}, \mathcal{M}, \mathcal{R}$, respectively, we define
\begin{equation*}
L^{\ell,m,r}=\sum_{\genfrac{}{}{0pt}{}{i=1,3,\,\delta_i>0}{\delta_i\in\mathcal{L},\mathcal{M},\mathcal{R}}}|\delta_i| + \xi \sum_{\genfrac{}{}{0pt}{}{i=1,3,\,\delta_i<0}{{\delta_i\in\mathcal{L},\mathcal{M},\mathcal{R}}}}|\delta_i|\,,
\end{equation*}
where $\xi>1$ is a parameter to be determined. We denote
$$
L= L^{\ell}+L^m+L^r\,,\qquad L^0=|\eta_0|+|\zeta_0|
$$
and
\begin{equation*}
\bar{L}  =  \bar{L}^\ell+\bar{L}^m+\bar{L}^r= \sum_{\genfrac{}{}{0pt}{}{i=1,3}{\delta_i\in\mathcal{L}}}|\delta_i|+\sum_{\genfrac{}{}{0pt}{}{i=1,3}{\delta_i\in\mathcal{M}}}|\delta_i|+\sum_{\genfrac{}{}{0pt}{}{i=1,3}{\delta_i\in\mathcal{R}}}|\delta_i|= \frac 12 \tv\left(\log p(t,\cdot)\right)-|\eta_0|-|\zeta_0|\,.
\end{equation*}
In the following sections we also introduce the interaction potential
\begin{equation}\label{eq:Q}
Q=Q^\ell+Q^m+Q^r
\end{equation}
and specify $Q^{\ell,m,r}$, that differ in the bubble case and in the increasing-pressure case. The resulting functional
\begin{equation}\label{eq:F}
F=L+L^0+Q
\end{equation}
is equivalent to the total variation of the approximate solutions and has an asymmetrical character in a double sense: firstly, it depends on the phases and, secondly, shocks and rarefactions play a different role. As for the latter, not only shocks are weighted by $\xi$ (a procedure of \cite{AmadoriGuerra01} also exploited in \cite{ABCD,ABCD2}), but the contributions of certain shock waves from $Q$ are dropped. This is due to the fact that in the interaction of some shocks with a $0$-wave we have that $\Delta L$ is already nonpositive; more precisely, from \eqref{eq:rels} it follows that such unnecessary waves are either $3$-shocks interacting with a $2$-wave $\delta>0$ or $1$-shocks interacting with a $2$-wave $\delta<0$. On the contrary, rarefaction waves are always counted in $Q$ if they approach a phase wave. As shown in Figure~\ref{fig:tv}, the total variation of the solutions increases going towards the more liquid regions.


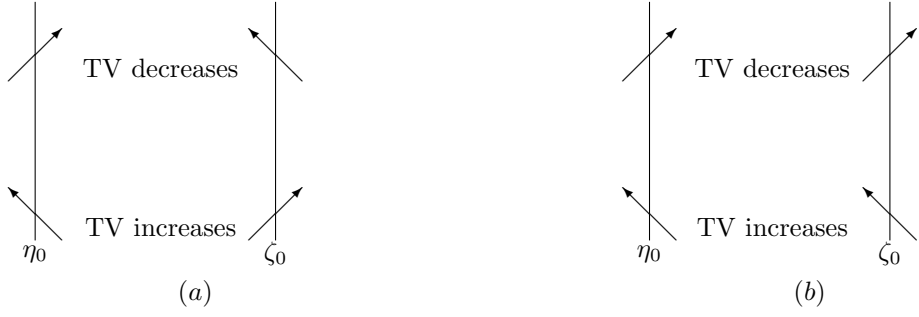
\begin{figure}[htbp]
\begin{picture}(100,110)(-120,10)
\setlength{\unitlength}{1pt}

\put(-80,0){
\put(10,30){\line(0,1){90}}
\put(100,30){\line(0,1){90}}

\put(0, 90){\vector(1,1){20}}
\put(110, 90){\vector(-1,1){20}}
\put(20, 30){\vector(-1,1){20}}
\put(90, 30){\vector(1,1){20}}

\put(10,25){\makebox(0,0){$\eta_0$}}
\put(100,25){\makebox(0,0){$\zeta_0$}}

\put(57,95){\makebox(0,0){$\tv$ decreases}}
\put(57,35){\makebox(0,0){$\tv$ increases}}

\put(70,10){\makebox(0,0){$(a)$}}
}

\put(150,0){
\put(10,30){\line(0,1){90}}
\put(100,30){\line(0,1){90}}

\put(0, 90){\vector(1,1){20}}
\put(90, 90){\vector(1,1){20}}
\put(20, 30){\vector(-1,1){20}}
\put(110, 30){\vector(-1,1){20}}

\put(10,25){\makebox(0,0){$\eta_0$}}
\put(100,25){\makebox(0,0){$\zeta_0$}}

\put(56,95){\makebox(0,0){$\tv$ decreases}}
\put(56,35){\makebox(0,0){$\tv$ increases}}

\put(70,10){\makebox(0,0){$(b)$}}
}
\end{picture}
\caption{\label{fig:tv}{How the total variation varies for interactions with the phase waves in the bubble case $(a)$ and in the increasing-pressure case $(b)$.}}
\end{figure}


In the potentials $Q^{\ell,m,r}$ we insert some positive weights $K_{\eta,\zeta}^{\ell,m,r}$ that keep track of the regions of provenience ($\mathcal{L},\mathcal{M},\mathcal{R}$) of the approaching waves and of the $0$-wave approached ($\eta_0$ or $\zeta_0$), see Figure~\ref{fig:paramK}.


\begin{figure}[htbp]
\begin{picture}(100,110)(-120,10)
\setlength{\unitlength}{1pt}

\put(25,0){
\put(10,30){\line(0,1){90}}
\put(150,30){\line(0,1){90}}
\put(-75, 54){\line(1,1){50}}
\put(-50, 46){\line(1,1){50}}
\put(71, 46){\line(-1,1){50}}
\put(91, 46){\line(1,1){50}}
\put(210, 46){\line(-1,1){50}}
\put(235, 54){\line(-1,1){50}}


\put(-40,70){\makebox(0,0){$K_\eta^\ell$}}
\put(-60,85){\makebox(0,0){$K_\zeta^\ell$}}
\put(55,80){\makebox(0,0){$K_\eta^m$}}
\put(105,80){\makebox(0,0){$K_\zeta^m$}}
\put(200,70){\makebox(0,0){$K_\zeta^r$}}
\put(220,85){\makebox(0,0){$K_\eta^r$}}
\put(10,25){\makebox(0,0){$\eta_0$}}
\put(150,25){\makebox(0,0){$\zeta_0$}}
}
\end{picture}
\vspace{-4ex}
\caption{\label{fig:paramK}{The parameters $K_{\eta,\zeta}^{\ell,m,r}$.}}
\end{figure}
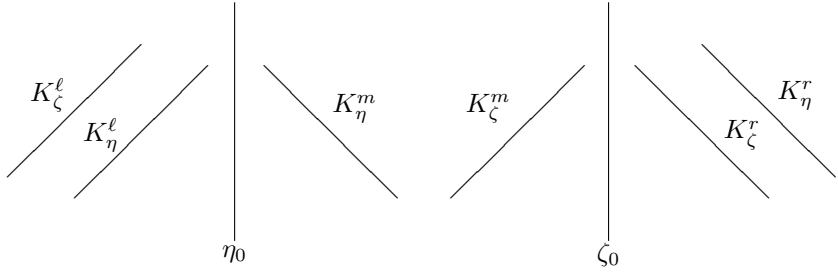


In the next sections we show that the functional $F$ decreases across any interaction, under suitable conditions on the parameters $\xi$, $K_{\eta,\zeta}^{\ell,m,r}$ and $\rho$, which differ in the bubble case and in the increasing-pressure case. In the proof we fix $m_o>0$ and assume that, for $i=1,3$, the size $\delta_i$ of any $i$-shock satisfies
\begin{equation}\label{rogna}
|\delta_i|\le m_o\,.
\end{equation}
Such parameter is necessary in our proof to relate the total variation of the initial data $u_o,p_o$ to the sizes of the phase waves.


\section{The bubble case}\label{sec:interbubble}
\setcounter{equation}{0}

In this section, we prove the decreasing of the functional $F$ in \eqref{eq:F} in the bubble case, where the phase-dependent interaction potentials are defined by
\begin{align*}
Q^\ell& =\left(K_\eta^\ell|\eta|+ K_\zeta^\ell|\zeta|\right)\sum_{\genfrac{}{}{0pt}{}{\delta_3>0}{\delta_3\in\mathcal{L}}}|\delta_3|+\xi K_\zeta^\ell\sum_{\genfrac{}{}{0pt}{}{\delta_3<0}{\delta_3\in\mathcal{L}}}|\delta_3\zeta|\,,	 \\
Q^{m}&=K_\eta^m\Bigl(\sum_{\genfrac{}{}{0pt}{}{\delta_1>0}{\delta_1\in\mathcal{M}}}|\delta_1\eta|+\xi \sum_{\genfrac{}{}{0pt}{}{\delta_1<0}{\delta_1\in\mathcal{M}}}|\delta_1\eta|\Bigr) + K_\zeta^m \Bigl(\sum_{\genfrac{}{}{0pt}{}{\delta_3>0}{\delta_3\in\mathcal{M}}}|\delta_3 \zeta|+\xi \sum_{\genfrac{}{}{0pt}{}{\delta_3<0}{\delta_3\in\mathcal{M}}}|\delta_3 \zeta|\Bigr)\,,\\
Q^r &= \left(K_\eta^r|\eta|+ K_\zeta^r|\zeta|\right)\sum_{\genfrac{}{}{0pt}{}{\delta_1>0}{\delta_1\in\mathcal{R}}}|\delta_1| +\xi K_\eta^r\sum_{\genfrac{}{}{0pt}{}{\delta_1<0}{\delta_1\in\mathcal{R}}}|\delta_1\eta|\,.
\end{align*}
As previously mentioned, the interaction potential $Q$ in \eqref{eq:Q} lacks certain shock waves: precisely, $3$-shocks interacting with $\eta_0$ and $1$-shocks interacting with $\zeta_0$, see Figure~\ref{fig:bubbleQ}. In the next two propositions we list the conditions on the parameters $\xi$, $K_{\eta,\zeta}^{\ell,m,r}$ and $\rho$ needed for the decrease of $F$.


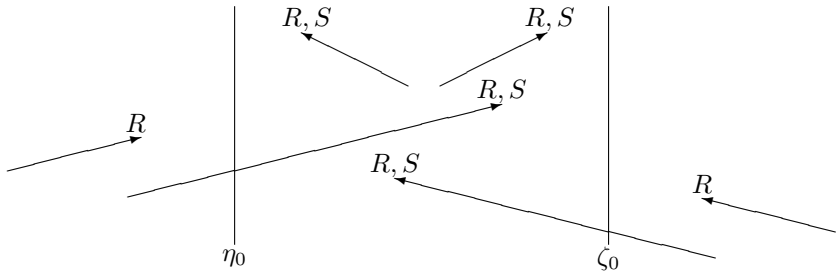
\begin{figure}[htbp]
\begin{picture}(100,110)(-120,10)
\setlength{\unitlength}{1pt}

\put(25,0){
\put(10,30){\line(0,1){90}}
\put(150,30){\line(0,1){90}}

\put(-75, 58){\vector(4,1){50}}
\put(-30, 48){\vector(4,1){140}}

\put(75, 90){\vector(-2,1){40}}
\put(87, 90){\vector(2,1){40}}

\put(190, 25){\vector(-4,1){120}}
\put(235, 35){\vector(-4,1){50}}

\put(-27,76){\makebox(0,0){$R$}}
\put(37,115){\makebox(0,0){$R,S$}}
\put(128,115){\makebox(0,0){$R,S$}}
\put(110,88){\makebox(0,0){$R,S$}}
\put(70,60){\makebox(0,0){$R,S$}}
\put(185,53){\makebox(0,0){$R$}}

\put(10,25){\makebox(0,0){$\eta_0$}}
\put(150,25){\makebox(0,0){$\zeta_0$}}
}
\end{picture}
\vspace{-4ex}
\caption{\label{fig:bubbleQ}{The waves considered in $Q$ for the bubble case.}}
\end{figure}


\begin{proposition}\label{prop:list1}
    Assume that at time $t>0$ a wave $\delta_i$, $i=1,3$, interacts with one of
    the composite waves $\eta_0$ or $\zeta_0$. Then, $\Delta F(t)\le 0$ provided
    that
    \begin{gather}
	\xi\ge 1, \qquad
	K_\zeta^m,K_\eta^m\geq1, \qquad
	\frac{\xi-1}{2}\le K_\zeta^r,K_\eta^\ell, \qquad
	K_\eta^m\le K_\eta^r, \qquad K_\zeta^m\le K_\zeta^\ell,\label{eq:composite1}\\
	 1+K_\eta^m\frac{|\eta|}{2}-K_\zeta^m\le 0,\qquad
	 1+K_\zeta^m\frac{|\zeta|}{2}-K_\eta^m\le 0,\label{eq:composite2}\\
	 C_{o}(\rho)\le\frac{2\xi}{\xi+1}\min\{K_\zeta^m,K_\eta^m\}.\label{eq:composite3}
    \end{gather}
\end{proposition}
\begin{proof}


\begin{figure}[htbp]
\begin{picture}(100,100)(-80,-15)
\setlength{\unitlength}{1pt}

\put(50,0){
\put(0,0){\line(0,1){40}} \put(-2,-5){\makebox(0,0){$\zeta_0$}}
\put(-60,0){\line(0,1){40}} \put(-62,-5){\makebox(0,0){$\eta_0$}}
\put(20,0){\line(-1,2){20}} \put(20,-5){\makebox(0,0){$\delta_1$}}
\put(0,40){\line(0,1){30}} \put(-2,75){\makebox(0,0){$\zeta_0$}}
\put(-60,40){\line(0,1){30}} 
\put(0,40){\line(1,2){15}} \put(20,75){\makebox(0,0){$\eps_3$}}
\put(0,40){\line(-1,1){30}} \put(-30,75){\makebox(0,0){$\eps_1$}}
\put(-37,40){\makebox(0,0){$\mathcal{M}$}}
\put(40,40){\makebox(0,0){$\mathcal{R}$}}
\put(-90,40){\makebox(0,0){$\mathcal{L}$}}
\put(-37,-22){\makebox(0,0){$(a)$}}
}

\put(270,0){
\put(0,0){\line(0,1){40}} \put(-2,-5){\makebox(0,0){$\zeta_0$}}
\put(-20,0){\line(1,2){20}}\put(-20,-5){\makebox(0,0){$\delta_3$}}
\put(-60,0){\line(0,1){40}} \put(-62,-5){\makebox(0,0){$\eta_0$}}
\put(-60,40){\line(0,1){30}} 
\put(0,40){\line(0,1){30}} \put(-2,75){\makebox(0,0){$\zeta_{0}$}}
\put(0,40){\line(1,1){30}} \put(30,75){\makebox(0,0){$\eps_3$}}
\put(0,40){\line(-1,2){15}} \put(-20,75){\makebox(0,0){$\eps_1$}}
\put(-37,40){\makebox(0,0){$\mathcal{M}$}}
\put(40,40){\makebox(0,0){$\mathcal{R}$}}
\put(-90,40){\makebox(0,0){$\mathcal{L}$}}
\put(-37,-22){\makebox(0,0){$(b)$}}
}

\end{picture}
\vspace{10pt}
\caption{\label{fig:inter2}{Interactions of $1$- and $3$-waves with $\zeta_0$ solved by means of the Accurate solver. The fronts carrying the composite waves are represented by a single line.}}
\end{figure}
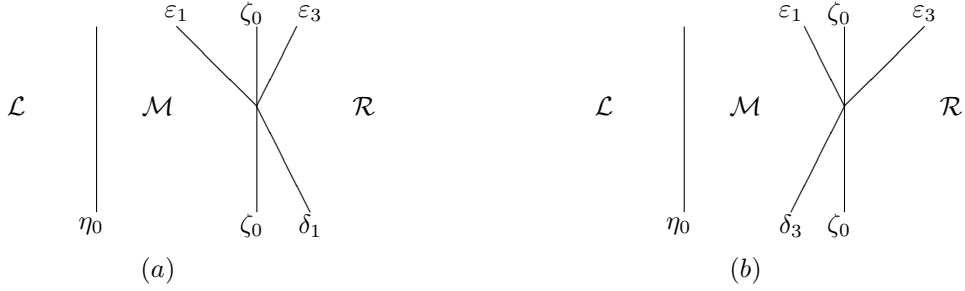


Since the two cases give symmetric conditions, we only analyze interactions involving $\zeta_0$; see Figure~\ref{fig:inter2}. 
By \eqref{eq:rels} and \eqref{eq:sign} we have
\begin{equation*}
\left\{
\begin{array}{lll}
\eps_3 -\eps_1= -\delta_1, &\quad |\eps_1| -|\delta_1| =  -|\eps_3|\,, &\qquad \mbox{if  }i=1, \\[1mm]
\eps_3 -\eps_1= \delta_3, &\quad |\eps_3|-|\delta_3| = |\eps_1|\,,  &\qquad \mbox{if  }i=3\,.

\end{array}
\right.
\end{equation*}
\smallskip\noindent{\fbox{$i=1$.}} If the interacting wave is a rarefaction,
then by \eqref{eq:intest1}
we have
$$
\Delta L+\Delta L^0=\begin{cases}
\ds\xi|\eps_3|+|\eps_1|-|\delta_1|=(\xi-1)|\eps_3|\le \frac{\xi-1}{2}|\delta_1\zeta| &\qquad \text{if $|\delta_1|\geq \rho$},\\[7pt]
\ds|\eps_3|+|\eps_1|-|\delta_1|=0 &\qquad \text{if $|\delta_1|<\rho$},
\end{cases}
$$
and $\Delta Q = K_\eta^m|\eps_1\eta|- K_\eta^r|\delta_1\eta|
-K_\zeta^r|\delta_1\zeta|$. Therefore, since $|\eps_1|\le|\delta_1|$ by the interaction estimates,
$$
\Delta F
\le \begin{cases}
\ds \left(K_\eta^m-K_\eta^r\right)|\delta_1\eta|+\bigl[\frac{\xi-1}{2}-K_\zeta^r\bigr]|\delta_1\zeta| &\qquad \text{if $|\delta_1|\ge\rho$},\\[7pt]
\ds \left(K_\eta^m-K_\eta^r\right)|\delta_1\eta|-K_\zeta^r|\delta_1\zeta| &\qquad \text{if $|\delta_1|<\rho$},
\end{cases}
$$
which is nonpositive by \eqref{eq:composite1}$_{3,4}$.
Instead, if the interacting wave is a shock, then in both the accurate and simplified cases we have $\Delta L+\Delta L^0=|\eps_3|+\xi|\eps_1|-\xi|\delta_1|=-(\xi-1)|\eps_3|$ and $\Delta Q=\xi(K_\eta^m|\eps_1|-K_\eta^r|\delta_1|)|\eta|$. As a consequence, $\Delta F\le -(\xi-1)|\eps_3|+\xi(K_\eta^m-K_\eta^r)|\delta_1\eta|$, which is $\le 0$ by \eqref{eq:composite1}$_{1,4}$.

\smallskip\noindent{\fbox{$i=3$.}} If the interacting wave is a rarefaction,
then by the interaction estimates \eqref{eq:intest1}
$\Delta L+ \Delta L^0=|\eps_3|+|\eps_1|-|\delta_3|=2|\eps_1|\le |\delta_3\zeta|$ and
$$
\Delta Q= \begin{cases}
K_\eta^m|\eps_1\eta|- K_\zeta^m|\delta_3\zeta| &\quad \text{if $|\delta_3|\ge \rho$},\\[7pt]
-K_\zeta^m|\delta_3\zeta| &\quad \text{if $|\delta_3|< \rho$}.
\end{cases}
$$
Then,
$$
\Delta F \le \begin{cases}
\ds \left[1+ K_\eta^m\frac{|\eta|}{2}-K_\zeta^m\right]|\delta_3\zeta| &\quad \text{if $|\delta_3|\ge \rho$},\\[7pt]
\ds \left[1-K_\zeta^m\right]|\delta_3\zeta| &\quad \text{if $|\delta_3|<\rho$},
\end{cases}
$$
which is nonpositive by \eqref{eq:composite1}$_{2}$, \eqref{eq:composite2}$_{1}$.
On the other hand, if the interacting wave is a shock,
then
$$
\Delta L+\Delta L^0= \begin{cases}
\ds \xi|\eps_1|+\xi|\eps_3|-\xi|\delta_3|=2\xi|\eps_1|\le \xi|\delta_3\zeta| &\qquad \text{if $|\delta_3|\ge \rho$},\\[7pt]
\ds |\eps_1|+\xi|\eps_3|-\xi|\delta_3|=(\xi+1)|\eps_1|\le (\xi+1)\frac{C_o}{2}|\delta_3\zeta| &\qquad \text{if $|\delta_3|<\rho$},
\end{cases}
$$
and
$$
\Delta Q=\begin{cases}
\ds \xi\left(K_\eta^m|\eps_1\eta|-K_\zeta^m|\delta_3\zeta|\right) &\quad \text{if $|\delta_3|\ge \rho$},\\[7pt]
-\xi K_\zeta^m|\delta_3\zeta| &\quad \text{if $|\delta_3|< \rho$}.
\end{cases}
$$
Therefore,
$$
\Delta F \le \begin{cases}
\ds \xi\left[1+K_\eta^m\frac{|\eta|}{2}-K_\zeta^m\right]|\delta_3\zeta| &\quad \text{if $|\delta_3|\ge \rho$},\\[7pt]
\ds \left[(\xi+1)\frac{C_o}{2}-\xi K_\zeta^m\right]|\delta_3\zeta| &\quad \text{if $|\delta_3|<\rho$},
\end{cases}
$$
which is nonpositive by \eqref{eq:composite2}$_{1}$, \eqref{eq:composite3}$_{1}$.
\end{proof}

\begin{proposition}\label{prop:list2}
    Consider the interaction at time $t>0$ of two waves of the same family $1$
    or $3$ and assume \eqref{rogna}. Then, $\Delta F(t)\le 0$ provided that
    \begin{gather}\label{cond13}
	1\le \xi \le \frac{1}{c(m_o)}, \qquad
	K_\zeta^m\le \frac{\xi-1}{|\zeta|},\qquad
	K_\eta^m\le \frac{\xi-1}{|\eta|},\\
	K_\eta^r|\eta|+K_\zeta^r|\zeta|\le \xi-1,\qquad
	K_\eta^\ell|\eta|+K_\zeta^\ell|\zeta|\le \xi-1.\label{cond13bis}
    \end{gather}
\end{proposition}

\begin{proof}
First, we consider the interactions taking place in $\mathcal{M}$, see Figure~\ref{fig:inter4}. For brevity, we only deal with the case of interactions between two $3$-waves $\alpha_3$ and $\beta_3$ giving rise to $\eps_1$ and $\eps_3$ (the $1$-waves case is analogous).


\begin{figure}[htbp]
\begin{picture}(100,80)(-130,-15)
\setlength{\unitlength}{0.8pt}

\put(130,0){
\put(-75,0){\line(0,1){40}} \put(-77,-5){\makebox(0,0){$\eta_0$}}
\put(-75,40){\line(0,1){30}} 
\put(55,0){\line(0,1){40}} \put(59,-5){\makebox(0,0){$\zeta_0$}}
\put(55,40){\line(0,1){30}} 
\put(0,40){\line(-1,-1){45}}\put(-55,0){\makebox(0,0){$\alpha_3$}}
\put(0,40){\line(-2,-3){30}}\put(-10,0){\makebox(0,0){$\beta_3$}}
\put(0,40){\line(1,2){15}} \put(20,75){\makebox(0,0){$\eps_3$}}
\put(0,40){\line(-1,1){30}} \put(-30,75){\makebox(0,0){$\eps_1$}}
\put(-100,40){\makebox(0,0){$\mathcal{L}$}}
\put(30,40){\makebox(0,0){$\mathcal{M}$}}
\put(90,40){\makebox(0,0){$\mathcal{R}$}}
}

\end{picture}

\caption{\label{fig:inter4}{Interactions of $3$-waves in $\mathcal{M}$.}}
\end{figure}
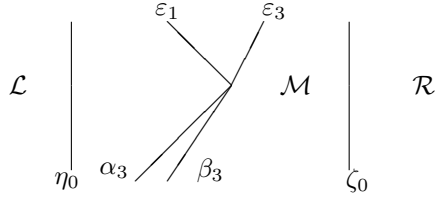


\noindent If both $\alpha_3$ and $\beta_3$ are shocks, then $\eps_1$ is a rarefaction by Lemma~\ref{lem:intest} and, as in \cite[Proposition 5.8]{ABCD}, we have
\begin{equation}\label{Delta_L_xi_13}
\Delta L+|\eps_1|(\xi-1)=0\,,
\end{equation}
for any $\xi\ge 1$. Moreover, we have
$$
\Delta Q=K_\eta^m|\eps_{1}\eta|-\xi K_\zeta^m|\eps_1\zeta|\le K_\eta^m|\eps_{1}\eta| \,,\qquad
\Delta F\le \left[-(\xi-1)+K_\eta^m|\eta|\right]|\eps_{1}|
$$
and $F$ is non-increasing by \eqref{cond13}$_{1,3}$.
On the other hand, when the two interacting waves are of different type,
for example $\alpha_3<0<\beta_3$, as in \cite[Proposition 5.8]{ABCD}, by \eqref{cond13}$_{1}$ one can deduce that
\begin{equation}\label{Delta_L_xi_13_SR}
\Delta L+\xi(\xi-1)|\eps_1|\le 0\,.
\end{equation}
If $\eps_3$ is a rarefaction, then $\Delta Q= \xi K_\eta^m |\eps_1\eta|+K_\zeta^m\left(|\eps_{3}|-\xi|\alpha_3|-|\beta_{3}|\right)|\zeta|$; on the other hand, if $\eps_3$ is a shock, $\Delta Q= \xi K_\eta^m |\eps_1\eta|+K_\zeta^m\left(\xi|\eps_{3}|-\xi|\alpha_3|-|\beta_{3}|\right)|\zeta|$. Therefore, by Lemma~\ref{lem:intest} in both cases it holds
$$
\Delta Q\le\xi K_\eta^m|\eps_1\eta|\,, \qquad \Delta F\le \xi \left[-(\xi-1)+K_\eta^m|\eta|\right]|\eps_1|
$$
and $F$ decreases by \eqref{cond13}$_{1,3}$. The analysis of the interactions between $1$-waves requires symmetrically the condition $K_\zeta^m\le(\xi-1)/|\zeta|$.

Next, we analyze the case of interactions taking place in $\mathcal{R}$, the case of interactions in $\mathcal{L}$ being analogous. By \eqref{cond13}$_{1}$ it is easy to verify that $F$ decreases when two $1$-waves interact.


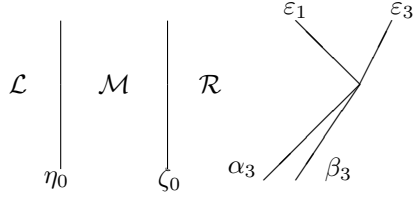
\begin{figure}[htbp]
\begin{picture}(100,80)(-130,-15)
\setlength{\unitlength}{0.8pt}

\put(190,0){
\put(-140,0){\line(0,1){40}} \put(-142,-5){\makebox(0,0){$\eta_0$}}
\put(-140,40){\line(0,1){30}} 
\put(-90,0){\line(0,1){40}} \put(-89,-5){\makebox(0,0){$\zeta_0$}}
\put(-90,40){\line(0,1){30}} 
\put(0,40){\line(-1,-1){45}}\put(-55,0){\makebox(0,0){$\alpha_3$}}
\put(0,40){\line(-2,-3){30}}\put(-10,0){\makebox(0,0){$\beta_3$}}
\put(0,40){\line(1,2){15}} \put(20,75){\makebox(0,0){$\eps_3$}}
\put(0,40){\line(-1,1){30}} \put(-30,75){\makebox(0,0){$\eps_1$}}
\put(-160,40){\makebox(0,0){$\mathcal{L}$}}
\put(-115,40){\makebox(0,0){$\mathcal{M}$}}
\put(-70,40){\makebox(0,0){$\mathcal{R}$}}
}

\end{picture}

\caption{\label{fig:inter5}{Interactions of $3$-waves in $\mathcal{R}$.}}
\end{figure}


\noindent Now, we consider the interactions between $3$-waves (see Figure~\ref{fig:inter5}). When the interacting waves $\alpha_3,\beta_3$ are both shocks we have \eqref{Delta_L_xi_13}, while in the other two cases estimate \eqref{Delta_L_xi_13_SR} still holds under condition \eqref{cond13}$_1$. If $\alpha_3,\beta_3<0$ we have
$$
\Delta Q=K_\eta^r|\eps_1\eta|+K_\zeta^r|\eps_{1}\zeta|\,,\qquad \Delta F= \left[-(\xi-1)+K_\eta^r|\eta|+K_\zeta^r|\zeta|\right]|\eps_1|\,,
$$
while if, for example, $\alpha_3<0<\beta_3$ we have
$$
\Delta Q =K_\eta^r\,\xi|\eps_1\eta|\,,\qquad
\Delta F\le \xi\left[-(\xi-1)+K_\eta^r|\eta|\right]|\eps_1|\,.
$$
Consequently, $F$ is non-increasing by \eqref{cond13bis}$_{1}$. The condition
$K_\eta^\ell|\eta|+K_\zeta^\ell|\zeta|\le \xi-1$ is required for the interactions occurring in the region $\mathcal{L}$.
\end{proof}


Now, we can determine the order of choice of the parameters. To simplify the analysis, we can let $K_\eta^m=K_\eta^r$ and $K_\zeta^m=K_\zeta^\ell$, since the final result does not change otherwise. Once $\eta,\zeta$ have been fixed, we choose in turn: $m_o$, $\xi$, $K_\eta^m$ and $K_\zeta^m$, $K_\zeta^r$ and $K_\eta^\ell$; finally, we choose $\rho$ (i.e.\ $C_o$). First, notice that the conditions in \eqref{eq:composite2} identify the set in the $(K_\eta^m,K_\zeta^m)$-plane represented in Figure~\ref{KmzKme}. Hence, by \eqref{eq:composite2} we deduce
$$
K_\eta^m \ge 1+K_\zeta^m\frac{|\zeta|}{2}\ge 1+\frac{|\zeta|}{2}\bigl(1+K_\eta^m\frac{|\eta|}{2}\bigr) \qquad \text{and} \qquad
K_\zeta^m \ge 1+K_\eta^m\frac{|\eta|}{2}\ge 1+\frac{|\eta|}{2}\bigl(1+K_\zeta^m\frac{|\zeta|}{2}\bigr),
$$
that imply
\begin{equation}\label{km}
K_\eta^m\ge \frac{1+|\zeta|/2}{1-|\eta\zeta|/4} \qquad \text{and} \qquad K_\zeta^m\ge \frac{1+|\eta|/2}{1-|\eta\zeta|/4}.
\end{equation}
In particular, by replacing the inequality sign by equality in \eqref{km} we get the coordinates of the intersection point $V$ between the two lines of Figure~\ref{KmzKme}. Notice also that \eqref{km} implies \eqref{eq:composite1}$_2$.


\begin{figure}[htbp]
    \begin{center}
	\includegraphics[width=5cm]{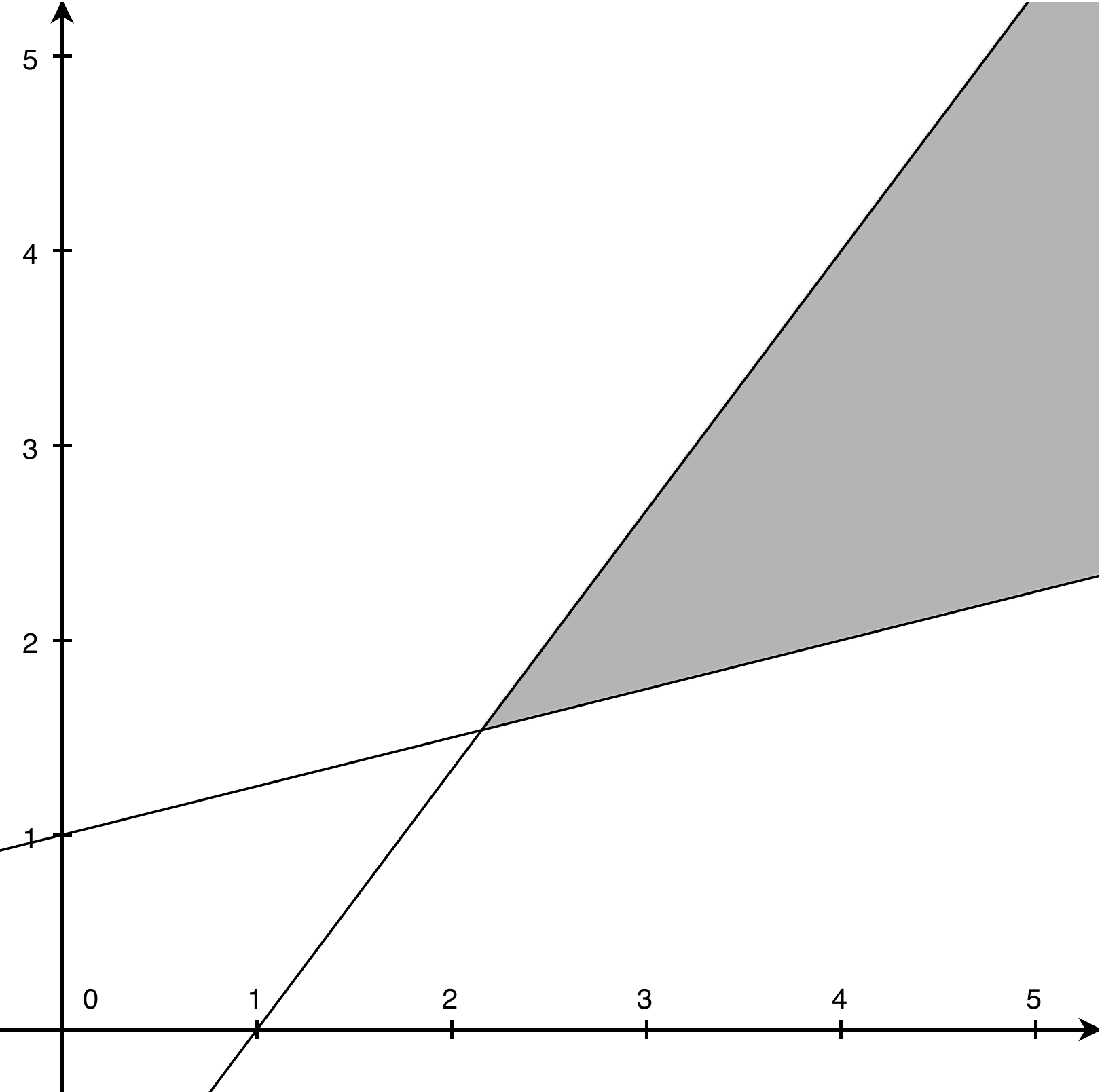}
	\begin{picture}(0,0)%
	    \put(-21,-5){$K_\eta^m$}%
		\put(-87,50){$V$}	
	    \put(-160,130){$K_\zeta^m$}%
	\end{picture}%
    \end{center}
    \caption{\label{KmzKme} Graphical representation of conditions \eqref{eq:composite2} for $|\eta|=1/2$ and $|\zeta|=3/2$.}
\end{figure}
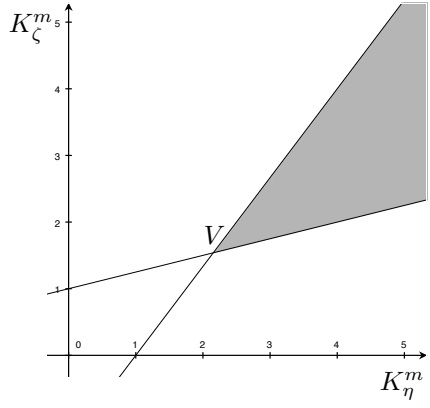


Since we have chosen $K_\eta^m=K_\eta^r$ and $K_\zeta^m=K_\zeta^\ell$, conditions \eqref{cond13bis} imply \eqref{cond13}$_{2,3}$. By \eqref{eq:composite1}$_3$ and \eqref{cond13bis}$_1$, we get $K_\eta^m|\eta|+(\xi-1)|\zeta|/2\le \xi-1$, which is equivalent to
\begin{equation}\label{eq:stella}
K_\eta^m\,\frac{|\eta|}{1-|\zeta|/2}\le \xi-1\,;
\end{equation}
similarly, by \eqref{eq:composite1}$_3$ and \eqref{cond13bis}$_2$ we get
\begin{equation}\label{eq:stella2}
K_\zeta^m\,\frac{|\zeta|}{1-|\eta|/2}\le \xi-1\,.
\end{equation}
By \eqref{km}, \eqref{eq:stella} and \eqref{eq:stella2} it follows that $\xi$ must satisfy the inequality
$$
\xi \ge 1+\max\left\{\frac{1+|\zeta|/2}{1-|\zeta|/2}\,\frac{|\eta|}{1-|\eta\zeta|/4},\,\frac{1+|\eta|/2}{1-|\eta|/2}\,\frac{|\zeta|}{1-|\eta\zeta|/4}\right\}\,.
$$
This condition must match with \eqref{cond13}$_1$; then, recalling \eqref{eq:b} we must require
\begin{equation}\label{eq:ximo0}
1+\mathcal{H}_b(|\eta|,|\zeta|)\le \xi \le \frac{1}{c(m_o)}\,,
\end{equation}
which is a condition that relates $m_o$ to $|\eta|,|\zeta|$. When one of the phase waves tends to zero, $\mathcal{H}_b(|\eta|,|\zeta|)$ tends to the other one and we completely recover the results of \cite{ABCD}.

Summarizing, we choose the parameters as follows and keep strict inequalities for later need; let $(|\eta|,|\zeta|)\in D_b$ be given.
\begin{itemize}
\item First, we fix $m_o$ such that
\begin{equation}\label{eq:ximo1}
1+\mathcal{H}_b(|\eta|,|\zeta|)<\frac{1}{c(m_o)}
\end{equation}
and take $\xi$ in the interior of the interval given by \eqref{eq:ximo0}.
\item In the $(K_\eta^m,K_\zeta^m)$-plane we choose a point in the affine cone defined by \eqref{eq:composite2} and sufficiently close to $V$; moreover, we require
\begin{equation}\label{eq:km2}
K_\eta^m\,\frac{|\eta|}{1-|\zeta|/2}< \xi-1 \qquad \text{and} \qquad K_\zeta^m\,\frac{|\zeta|}{1-|\eta|/2}< \xi-1\,.
\end{equation}
\item We choose $K_\eta^r=K_\eta^m$, $K_\zeta^\ell=K_\zeta^m$ and, then, by \eqref{eq:km2} we choose $K_\zeta^r$ and $K_\eta^\ell$ such that
\begin{equation}\label{eq:km3}
\frac{\xi-1}{2}<K_\zeta^r<\frac{\xi-1}{|\zeta|}-K_\eta^m\frac{|\eta|}{|\zeta|} \qquad \text{and} \qquad \frac{\xi-1}{2}<K_\eta^\ell<\frac{\xi-1}{|\eta|}-K_\zeta^m\frac{|\zeta|}{|\eta|}\,.
\end{equation}
\item Finally, we choose $\rho$ such that \eqref{eq:composite3} holds.
\end{itemize}

\noindent Now, we can prove the global in time decreasing of the functional $F$.

\begin{proposition}\label{prop:lastb}
Let $m_o > 0$ satisfy \eqref{eq:ximo1}. Moreover, assume that $\xi$, $K_{\eta,\zeta}^{\ell,m,r}$ and $\rho$ satisfy \eqref{eq:ximo0}--\eqref{eq:km3} and \eqref{eq:composite3}. Then, the following two statements hold.
\begin{itemize}
\item[i)] \emph{Local Decreasing}. For any interaction at time $t>0$ between two waves satisfying \eqref{rogna}, it holds
\begin{equation*}
\Delta F(t) \le 0\,.
\end{equation*}
\item[ii)] \emph{Global Decreasing}. If
\begin{equation}
	\label{eq:boundLb}
	\bar{L}(0)\le m_o\hspace{0.8pt} c(m_o)
\end{equation}
and the approximate solution is defined in $[0,T]$, then $F(0)\le m_o$, $\Delta F(t)\le 0$ for every $t\in(0,T]$ and \eqref{rogna} is satisfied.
\end{itemize}
\end{proposition}
\begin{proof}
The first statement has been proved above. As for the second assertion, let us denote by $L^{\ell,m,r}_{iR}$ and $L^{\ell,m,r}_{iS}$ the partial sums in $L^{\ell,m,r}$ due to $i$-rarefaction waves ($iR$) and $i$-shock waves ($iS$), respectively. By \eqref{cond13}$_{2,3}$ we have
\begin{align*}
F^m(0)&= L^m(0)+Q^m(0)\le L^m(0)\left(1+\max\{K_\eta^m|\eta|,K_\zeta^m|\zeta|\}\right)\le \xi^2 \bar{L}^{m}(0)\,.
\end{align*}
Moreover, from \eqref{cond13bis} it follows
\begin{align*}
F^{\ell}(0)&\le L^{\ell}_{1R}(0)+L^{\ell}_{1S}(0)+L^{\ell}_{3R}(0)\left(1+K_\eta^\ell|\eta|+ K_\zeta^\ell|\zeta|\right)+L^\ell_{3S}(0)\left(1+K_\zeta^\ell|\zeta|\right)\le \xi^2\bar{L}^{\ell}(0)\,,\\
F^{r}(0)&\le L^{r}_{3R}(0)+L^{r}_{3S}(0)+L^{r}_{1R}(0)\left(1+K_\eta^r|\eta|+ K_\zeta^r|\zeta|\right)+L^r_{1S}(0)\left(1+K_\eta^r|\eta|\right)\le \xi^2\bar{L}^{r}(0)\,.
\end{align*}
Then,
\begin{align*}
    F(0)&=F^\ell(0)+F^m(0)+F^r(0)\le \xi^2\bar{L}(0)\,.
\end{align*}
For a fixed $t\le T$, suppose by induction that $F(\tau)\le m_o$ and $\Delta F(\tau)\le 0$ for every $0<\tau<t$, interaction time. Then, the inequality $\Delta F(t)\le 0$ implies that
$$
F(t)\le F(0)\le \xi^2\bar{L}(0)\,.
$$
Hence, by \eqref{eq:boundLb} the size of a shock $\delta_i$ ($i=1,3$) at time $t$ satisfies
$$
|\delta_i|\le \frac{1}{\xi}F(t)\le \xi \bar L(0)\le \frac{1}{c(m_o)}\bar L(0)\le m_o
$$
and \eqref{rogna} is verified.
\end{proof}


\section{The increasing-pressure case}\label{sec:interdstep}
\setcounter{equation}{0}

In this section, we prove the decreasing of the functional $F$ in \eqref{eq:F} in the increasing-pressure case. As before, we first introduce the following interaction potentials
\begin{align*}
Q^\ell& =\left(K_\eta^\ell|\eta|+ K_\zeta^\ell|\zeta|\right)\sum_{\genfrac{}{}{0pt}{}{\delta_3>0}{\delta_3\in\mathcal{L}}}|\delta_3|\,,	 \\
Q^{m}&=K_\eta^m|\eta|\Bigl(\sum_{\genfrac{}{}{0pt}{}{\delta_1>0}{\delta_1\in\mathcal{M}}}|\delta_1|+\xi \sum_{\genfrac{}{}{0pt}{}{\delta_1<0}{\delta_1\in\mathcal{M}}}|\delta_1|\Bigr) + K_\zeta^m|\zeta|\sum_{\genfrac{}{}{0pt}{}{\delta_3>0}{\delta_3\in\mathcal{M}}}|\delta_3|\,,\\
Q^r &= \left(K_\eta^r|\eta|+K_\zeta^r|\zeta|\right)\Bigl(\sum_{\genfrac{}{}{0pt}{}{\delta_1>0}{\delta_1\in\mathcal{R}}}|\delta_1| +\xi\sum_{\genfrac{}{}{0pt}{}{\delta_1<0}{\delta_1\in\mathcal{R}}}|\delta_1|\Bigr)\,.
\end{align*}
The interaction potential $Q$ in \eqref{eq:Q} lacks the $3$-shocks interacting with $\eta_0$ and with $\zeta_0$, see Figure~\ref{fig:incpresQ}. The next proposition, which is analogous to Proposition~\ref{prop:list1}, gives a first list of conditions that guarantee the decrease of $F$.


\begin{figure}[htbp]
\begin{picture}(100,110)(-120,10)
\setlength{\unitlength}{1pt}

\put(25,0){
\put(10,30){\line(0,1){90}}
\put(150,30){\line(0,1){90}}

\put(-75, 58){\vector(4,1){50}}
\put(-30, 48){\vector(4,1){140}}

\put(75, 90){\vector(-2,1){40}}
\put(86, 90){\vector(2,1){40}}

\put(190, 25){\vector(-4,1){120}}
\put(235, 35){\vector(-4,1){50}}

\put(-27,76){\makebox(0,0){$R$}}
\put(37,115){\makebox(0,0){$R,S$}}
\put(128,115){\makebox(0,0){$R$}}
\put(110,88){\makebox(0,0){$R$}}
\put(70,60){\makebox(0,0){$R,S$}}
\put(185,53){\makebox(0,0){$R,S$}}

\put(10,25){\makebox(0,0){$\eta_0$}}
\put(150,25){\makebox(0,0){$\zeta_0$}}
}
\end{picture}
\vspace{-4ex}
\caption{\label{fig:incpresQ}{The waves considered in $Q$ for the increasing-pressure case.}}
\end{figure}
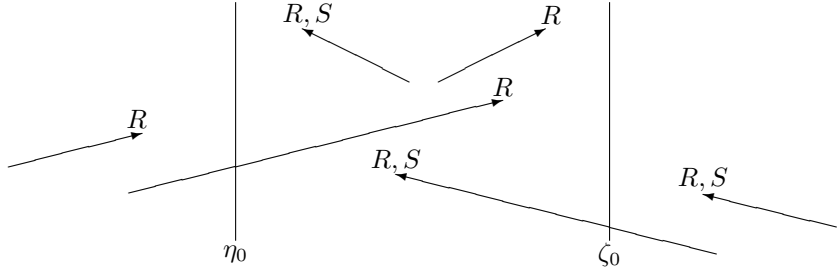


\begin{proposition}\label{prop:list3}
Assume that at time $t>0$ a wave $\delta_i$, $i=1,3$, interacts with one of the composite waves $\eta_0$ or $\zeta_0$. Then, $\Delta F(t)\le 0$ provided that
    \begin{gather}
	\xi\ge 1\,, \qquad K_\eta^m\ge 1\,, \qquad K_\eta^m\le \frac{\xi-1}{|\eta|}\,,\label{eq:composite4}\\ \bigl(\frac{\xi-1}{2}-K_\eta^\ell\bigr)|\eta|+(K_\zeta^m-K_\zeta^\ell)|\zeta|\le 0\,, \qquad
	 (1-K_\zeta^r)|\zeta|+\left(K_\eta^m\bigl(1+\frac{|\zeta|}{2}\bigr)-K_\eta^r\right)|\eta|\le 0\,,\label{eq:composite5}\\
	\frac{\xi-1}{2}+K_\eta^m\xi\frac{|\eta|}{2}-K_\zeta^m\le 0,\qquad 1+K_\zeta^m\frac{|\zeta|}{2}- K_\eta^m\le 0\,,\label{eq:composite6} \\
	\left((\xi+1)\frac{C_o}{2}-\xi K_\zeta^r\right)|\zeta|+\xi \left(K_\eta^m\bigl(1+\frac{C_o}{2}|\zeta|\bigr)-K_\eta^r\right)|\eta|\le 0\,,\qquad (\xi+1)\frac{C_o}{2}-\xi K_\eta^m\le 0\,. \label{eq:composite7}
    \end{gather}
\end{proposition}
\begin{proof}
Both in the case of interaction with $\eta_0$ and $\zeta_0$, by \eqref{eq:rels}, \eqref{eq:sign} we have
\begin{equation*}
\left\{
\begin{array}{lll}
\eps_3 -\eps_1= -\delta_1, &\quad |\eps_1| -|\delta_1| = |\eps_3|\,, &\qquad \mbox{if  }i=1, \\[1mm]
\eps_3 -\eps_1= \delta_3, &\quad |\eps_3|-|\delta_3| = -|\eps_1|\,,  &\qquad \mbox{if  }i=3\,.
\end{array}
\right.
\end{equation*}
However, in this case we have to treat separately the interactions with the phase waves, since the interaction potential $Q$ is not symmetric with respect to $\eta_0$ and $\zeta_0$.

\smallskip\noindent{\fbox{Interactions with $\eta_0$.}} Assume $i=1$. If $\delta_1$ is a rarefaction, then $\Delta L+\Delta L^0=2|\eps_3|\le |\delta_1\eta|$ and
$$
\Delta Q=\begin{cases}
K_\zeta^m|\eps_3\zeta|-K_\eta^m|\delta_1\eta| &\text{if $|\delta_1|\ge \rho$,}\\[7pt]
-K_\eta^m|\delta_1\eta| &\text{if $|\delta_1|<\rho$.}
\end{cases}
$$
Hence,
$$
\Delta F \le \begin{cases}
\ds \left[1+K_\zeta^m\frac{|\zeta|}{2}- K_\eta^m\right]|\delta_1\eta| &\text{if $|\delta_1|\ge\rho$,}\\[7pt]
\left[1- K_\eta^m\right]|\delta_1\eta| &\text{if $|\delta_1|<\rho$.}
\end{cases}
$$
Then, $\Delta F\le 0$ by \eqref{eq:composite6}$_{2}$ and \eqref{eq:composite4}$_2$. On the other hand, if $\delta_1$ is a shock then
$$
\Delta L+ \Delta L^0=\begin{cases}
\ds 2\xi|\eps_3|\le \xi|\delta_1\eta| &\text{if $|\delta_1|\ge \rho$,}\\[7pt]
\ds (\xi+1)|\eps_3|\le \frac{C_o}{2}(\xi+1)|\delta_1\eta| &\text{if $|\delta_1|<\rho$}
\end{cases}
$$
and $\Delta Q=-\xi K_\eta^m|\delta_1\eta|$. Hence,
$$
\Delta F\le \begin{cases}
\ds \xi\left[1-K_\eta^m\right]|\delta_1\eta| &\text{if $|\delta_1|\ge \rho$,}\\[7pt]
\ds \left[(\xi+1)\frac{C_o}{2}-\xi K_\eta^m\right]|\delta_1\eta|, &\text{if $|\delta_1|<\rho$,}
\end{cases}
$$
which is nonpositive by \eqref{eq:composite4}$_2$ and \eqref{eq:composite7}$_2$.

Now, let $i=3$. If $\delta_3$ is a rarefaction, then
$$
\Delta L +\Delta L^0=\begin{cases}
\ds(\xi-1)|\eps_1|\le \frac{\xi-1}{2}|\delta_3\eta| &\text{if $|\delta_3|\ge \rho$,}\\[7pt]
\ds 0 &\text{if $|\delta_3|<\rho$,}
\end{cases}
$$
and $\Delta Q=K_\zeta^m|\eps_3\zeta|-K_\eta^\ell|\delta_3\eta|-K_\zeta^\ell|\delta_3\zeta|$. Then, since $|\eps_3|\le|\delta_3|$ we have
$$
\Delta F\le \begin{cases}
\ds \left[\bigl(\frac{\xi-1}{2}-K_\eta^\ell\bigr)|\eta|+(K_\zeta^m-K_\zeta^\ell)|\zeta|\right]|\delta_3| &\text{if $|\delta_3|\ge \rho$,}\\[7pt]
\ds (K_\zeta^m- K_\zeta^\ell)|\delta_3\zeta|-K_\eta^\ell|\delta_3\eta| &\text{if $|\delta_3|<\rho$,}
\end{cases}
$$
and $F$ decreases by \eqref{eq:composite5}$_1$. If $\delta_3$ is a shock, then in any case we get $\Delta F= -(\xi-1)|\eps_1|\le 0$ by \eqref{eq:composite4}$_1$.

\smallskip\noindent{\fbox{Interactions with $\zeta_0$.}} Assume $i=1$. If $\delta_1$ is a rarefaction, then $\Delta L+\Delta L^0=2|\eps_3|\le |\delta_1\zeta|$ and $\Delta Q=K_\eta^m|\eps_1\eta|-K_\eta^r|\delta_1\eta|-K_\zeta^r|\delta_1\zeta|$ in both the Accurate and the Simplified case. Hence, by \eqref{eq:composite5}$_2$ we have
$$
\Delta F \le \left[(1-K_\zeta^r)|\zeta|+\left(K_\eta^m\bigl(1+\frac{|\zeta|}{2}\bigr)-K_\eta^r\right)|\eta|\right]|\delta_1| \le 0.
$$
On the other hand, if $\delta_1$ is a shock, then
$$
\Delta L+\Delta L^0=\begin{cases}
\ds 2\xi|\eps_3|\le \xi|\delta_1\zeta| &\text{if $|\delta_1|\ge \rho$,}\\[7pt]
\ds (\xi+1)|\eps_3|\le \frac{C_o}{2}(\xi+1)|\delta_1\zeta|  &\text{if $|\delta_1|<\rho$,}
\end{cases}
$$
and $\Delta Q=\xi K_\eta^m|\eps_1\eta|-\xi K_\eta^r|\delta_1\eta|-\xi K_\zeta^r|\delta_1\zeta|$. Thus,
$$
\Delta F\le \begin{cases}
\ds \xi \left[(1-K_\zeta^r)|\zeta|+\left(K_\eta^m\bigl(1+\frac{|\zeta|}{2}\bigr)-K_\eta^r\right)|\eta|\right]|\delta_1| &\text{if $|\delta_1|\ge \rho$,}\\[7pt]
\ds \left[\left((\xi+1)\frac{C_o}{2}-\xi K_\zeta^r\right)|\zeta|+\xi \left(K_\eta^m\bigl(1+\frac{C_o}{2}|\zeta|\bigr)-K_\eta^r\right)|\eta|\right]|\delta_1| &\text{if $|\delta_1|<\rho$,}
\end{cases}
$$
which is nonpositive by \eqref{eq:composite5}$_2$ and \eqref{eq:composite7}$_{1}$.

Assume $i=3$. If $\delta_3$ is a rarefaction, then
$$
\Delta L+\Delta L^0=\begin{cases}
\ds(\xi-1)|\eps_1|\le \frac{\xi-1}{2}|\delta_3\zeta| &\text{if $|\delta_3|\ge \rho$,}\\[7pt]
\ds 0 &\text{if $|\delta_3|<\rho$,}
\end{cases}
$$
and
$$
\Delta Q=\begin{cases}
\ds \xi K_\eta^m|\eps_1\eta|-K_\zeta^m|\delta_3\zeta| &\text{if $|\delta_3|\ge \rho$,}\\[7pt]
\ds -K_\zeta^m|\delta_3\zeta| &\text{if $|\delta_3|<\rho$.}
\end{cases}
$$
Hence,
$$
\Delta F\le \begin{cases}
\ds \left[\frac{\xi-1}{2}+K_\eta^m\xi\frac{|\eta|}{2}-K_\zeta^m\right]|\delta_3\zeta| &\text{if $|\delta_3|\ge \rho$,}\\[7pt]
\ds -K_\zeta^m|\delta_3\zeta| &\text{if $|\delta_3|<\rho$,}
\end{cases}
$$
which is nonpositive by \eqref{eq:composite6}$_1$. If, instead, $\delta_3$ is a shock, then $\Delta L+\Delta L^0=-(\xi-1)|\eps_1|$ and
$$
\Delta Q=\begin{cases}
K_\eta^m|\eps_1\eta| &\text{if $|\delta_3|\ge \rho$,}\\[7pt]
0 &\text{if $|\delta_3|<\rho$.}
\end{cases}
$$
Then,
$$
\Delta F\le \begin{cases}
-(\xi-1)|\eps_1|+K_\eta^m|\eps_1\eta| &\text{if $|\delta_3|\ge \rho$,}\\[7pt]
-(\xi-1)|\eps_1| &\text{if $|\delta_3|<\rho$,}
\end{cases}
$$
which is nonpositive by \eqref{eq:composite4}$_{1,3}$.
\end{proof}

As for an interaction between two waves of the same family, Proposition~\ref{prop:list2} still holds with the current functional $F$. Therefore, the conditions required on the various parameters are \eqref{cond13}, \eqref{cond13bis}. We omit the proof, since it can be carried out as above.

%

\bigskip

Here, we make some comments on the conditions \eqref{eq:composite4}--\eqref{eq:composite7} and \eqref{cond13}, \eqref{cond13bis}; finally, we establish the order in which we can choose the parameters.

First, notice that \eqref{eq:composite4}$_2$ is implied by \eqref{eq:composite6}$_2$. Secondly, we can rewrite \eqref{eq:composite5} as
\begin{equation}\label{eq:1s}
\frac{\xi-1}{2}|\eta|+K_\zeta^m|\zeta| \le K_\eta^\ell|\eta|+ K_\zeta^\ell|\zeta|\,, \qquad K_\eta^m\bigl(1+\frac{|\zeta|}{2}\bigr)|\eta|+|\zeta| \le K_\eta^r|\eta|+ K_\zeta^r|\zeta|\,.
\end{equation}
Putting together \eqref{eq:1s}$_2$ with \eqref{cond13bis}$_2$ and \eqref{eq:1s}$_1$ with \eqref{cond13bis}$_1$, we have
\begin{align}
K_\eta^m\bigl(1+\frac{|\zeta|}{2}\bigr)|\eta|+|\zeta| &\le K_\eta^r|\eta|+ K_\zeta^r|\zeta| \le \xi-1\,,\label{eq:4s}\\
\frac{\xi-1}{2}|\eta|+K_\zeta^m|\zeta|& \le K_\eta^\ell|\eta|+ K_\zeta^\ell|\zeta| \le \xi-1\,, \label{eq:3s}
\end{align}
then \eqref{cond13}$_{2,3}$ are implied by \eqref{eq:composite5} and \eqref{cond13bis}. Moreover, by \eqref{eq:composite6} we have
\begin{equation}\label{eq:2s}
K_\eta^m\ge 1+K_\zeta^m\frac{|\zeta|}{2}, \qquad K_\zeta^m\ge \frac{\xi-1}{2}+\xi K_\eta^m\frac{|\eta|}{2},
\end{equation}
that give the following lower bounds on $K_\eta^m$ and $K_\zeta^m$:
\begin{equation}\label{eq:5s}
K_\eta^m \ge \frac{1+(\xi-1)|\zeta|/4}{1-\xi|\eta\zeta|/4}, \qquad K_\zeta^m\ge \frac{(\xi-1)+\xi|\eta|}{2(1-\xi|\eta\zeta|/4)}.
\end{equation}
Remark that \eqref{eq:2s} represents an affine cone in the $(K_\eta^m,K_\zeta^m)$-plane under the condition
\begin{equation}\label{eq:7s}
\xi \le \frac{4}{|\eta\zeta|}\,.
\end{equation}
The vertex is the point whose coordinates are given by \eqref{eq:5s}. Hence, $K_\eta^m$ and $K_\zeta^m$ must be chosen in the non-empty intervals identified by \eqref{eq:4s}, \eqref{eq:5s}$_1$ and \eqref{eq:3s}, \eqref{eq:5s}$_2$, respectively. This means that \eqref{eq:3s}, \eqref{eq:5s}$_2$ give the condition
\begin{equation}\label{eq:8s}
\frac{\xi-1}{2}|\eta|+\frac{(\xi-1)+\xi|\eta|}{2(1-\xi|\eta\zeta|/4)}|\zeta|\le \xi-1,
\end{equation}
while \eqref{eq:4s}, \eqref{eq:5s}$_1$ give the condition
\begin{equation}\label{eq:9s}
\frac{1+(\xi-1)|\zeta|/4}{1-\xi|\eta\zeta|/4}\bigl(1+\frac{|\zeta|}{2}\bigr)|\eta|+|\zeta|\le \xi-1.
\end{equation}
We introduce the notation $|\eta|=x,|\zeta|=y$ and $\xi-1=z$. Then, by \eqref{eq:7s} we rewrite \eqref{eq:8s}  and \eqref{eq:9s} as, respectively,
\begin{align}
\frac{xy}{4}(2-x)z^2+\bigl[y(x+1)-(2-x)\bigl(1-\frac{xy}{4}\bigr)\bigr]z+xy &\le 0\,, \label{eq:new8}\\
\frac{xy}{4}z^2+\bigl[\frac{xy}{8}(4-y)-1\bigr]z +\bigl(1+\frac{y}{2}\bigr)x+y\bigl(1-\frac{xy}{4}\bigr) &\le 0.\label{eq:new9}
\end{align}
We also denote
$a(x,y)= xy(2-x)/4$, $b(x,y)=y(x+1)-(2-x)\left(1-xy/4\right)$, $c(x,y)=xy$, $d(x,y)=xy/4$, $e(x,y)=xy(4-y)/8-1$ and $f(x,y)=\left(1+y/2\right)x+y\left(1-xy/4\right)$ so that \eqref{eq:new8} and \eqref{eq:new9} become, respectively,
\begin{align}
a(x,y)z^2+b(x,y)z+c(x,y) &\le 0\,,\label{eq:newnew8} \\
d(x,y)z^2+e(x,y)z+f(x,y) &\le 0\,. \label{eq:newnew9}
\end{align}
Notice that the coefficients $a,c,d,f$ are positive, $e$ is
negative 
and $b$ may change sign. In order that each equations associated to \eqref{eq:newnew8} and \eqref{eq:newnew9} have distinct solutions, the discriminants $b^2-4ac$ and $e^2-4df$ must be strictly positive. If $b<0$, such solutions are positive.
Thus, about \eqref{eq:new8} we require
\begin{equation}\label{eq:dominio8}
y(x+1)-(2-x)\bigl(1-\frac{xy}{4}\bigr)+xy\sqrt{2-x}<0\,,
\end{equation}
while about \eqref{eq:new9} we impose
\begin{equation}\label{eq:dominio9}
\bigl[\frac{xy}{8}(4-y)-1\bigr]^2-xy\bigl[\bigl(1+\frac{y}{2}\bigr)x+y\bigl(1-\frac{xy}{4}\bigr)\bigr]>0\,.
\end{equation}
By a numerical comparison, we see that the set defined by \eqref{eq:dominio8} is included in that defined by \eqref{eq:dominio9}.

%
%

\noindent Under \eqref{eq:dominio8} and \eqref{eq:dominio9}, we denote by
\begin{equation}\label{radici8}
z_{1,2}(x,y)=\frac{(2-x)(1-xy/4)-y(x+1)\pm\sqrt{[y(x+1)-(2-x)(1-xy/4)]^2-x^2y^2(2-x)}}{xy(2-x)/2}
\end{equation}
the two positive solutions of the equation associated to \eqref{eq:newnew8} and by
\begin{equation}\label{radici9}
z_{3,4}(x,y)=\frac{1-xy(4-y)/8\pm \sqrt{[xy(4-y)/8 - 1]^2-xy[(1+y/2)x+y(1-xy/4)]}}{xy/2}
\end{equation}
the solutions of the equation associated to \eqref{eq:newnew9}. Hence, by \eqref{eq:7s},  \eqref{eq:new8} and \eqref{eq:new9} we get
\begin{equation}\label{eq:cxis}
1+\max\left\{z_1(x,y),z_3(x,y)\right\}< \xi < 1+\min\left\{z_2(x,y),z_4(x,y),\frac{4}{xy}-1\right\}\,.
\end{equation}
Therefore, we can define the domain $D_c$ represented in Figure~\ref{fig:Dc} as
\begin{equation}\label{eq:dc}
D_c=\left\{(|\eta|,|\zeta|)=(x,y): \text{\eqref{eq:cxis} holds} \right\}
\end{equation}
and the function
\begin{equation}\label{eq:c}
\mathcal{H}_c(|\eta|,|\zeta|)=\max\left\{z_1(|\eta|,|\zeta|),z_3(|\eta|,|\zeta|)\right\}\,.
\end{equation}
By \eqref{cond13}$_1$ we find the condition that relates $m_o$ to $|\eta|,|\zeta|$, i.e.\
\begin{equation}\label{eq:cximos}
1+\mathcal{H}_c(|\eta|,|\zeta|)<\frac{1}{c(m_o)}\,.
\end{equation}
As a final remark, we notice that \eqref{eq:composite7}$_1$ is equivalent to
\begin{equation}\label{eq:composite7bis}
\left(\frac{\xi+1}{2\xi}C_o-K_\zeta^r\right)|\zeta|+\left(K_\eta^m\bigl(1+\frac{C_o}{2}|\zeta|\bigr)-K_\eta^r\right)|\eta|\le 0\,.
\end{equation}
Then, by taking $\rho$ sufficiently small (since $C_o(\rho)\to 1$ if $\rho\to 0^+$) and $\xi>1$, \eqref{eq:composite7bis} is implied by \eqref{eq:composite5}$_2$.

\smallskip

For the choice of the parameters we proceed as follows.
\begin{itemize}
    \item We fix  $|\eta|,|\zeta|$ such that
    \begin{equation}\label{eq:domains}
    1+\max\left\{z_1(|\eta|,|\zeta|),z_3(|\eta|,|\zeta|)\right\}<1+\min\left\{z_2(|\eta|,|\zeta|),z_4(|\eta|,|\zeta|),\frac{4}{|\eta\zeta|}-1\right\}\,.
    \end{equation}
    Then, we fix $m_o$ such that \eqref{eq:cximos} hold and, in turn, we choose $\xi$ satisfying both \eqref{eq:cxis} and
    \begin{equation}\label{eq:cximo2s}
    1+\mathcal{H}_c(|\eta|,|\zeta|)<\xi<\frac{1}{c(m_o)}\,,
    \end{equation}
    so that \eqref{cond13}$_1$ holds.
	\item We choose $K_\eta^m,K_\zeta^m$ such that \eqref{eq:composite6} holds; in particular, we take $(K_\eta^m,K_\zeta^m)$ sufficiently close to the vertex of the cone and satisfying \eqref{eq:4s}--\eqref{eq:5s}, i.e.\ such that
    \begin{equation}\label{eq:kmezs}
    \frac{1+(\xi-1)|\zeta|/4}{1-\xi|\eta\zeta|/4}\le K_\eta^m<
    \frac{\xi-1-|\zeta|}{(1+{|\zeta|}/{2})|\eta|}\,,
    \qquad
    \frac{(\xi-1)+\xi|\eta|}{2(1-\xi|\eta\zeta|/4)}\le K_\zeta^m<
    \frac{\xi-1-(\xi-1)|\eta|/2}{|\zeta|}\,.
    \end{equation}
    Then, we choose $K_\eta^\ell=K_\zeta^\ell$, $K_\eta^r=K_\zeta^r$ such that
    \begin{gather}\label{eq:klrezs}
	\frac{(\xi-1)|\eta|/2+K_\zeta^m|\zeta|}{|\eta|+|\zeta|}
	\le K_\eta^\ell=K_\zeta^\ell<\frac{\xi-1}{|\eta|+|\zeta|}\,,\qquad
	\frac{K_m ^\eta\bigl(1+|\zeta|/2\bigr)|\eta|+|\zeta|}{|\eta|+|\zeta|}	
	\le K_\eta^r=K_\zeta^r<\frac{\xi-1}{|\eta|+|\zeta|}\,.
    \end{gather}
    Thus, \eqref{cond13bis} and \eqref{eq:composite5} hold; hence, also \eqref{cond13}$_{2,3}$ are verified.
    \item Finally, we choose $\rho$ such that $C_o(\rho)$ verifies \eqref{eq:composite7}.
\end{itemize}

In the next proposition, we claim the global in time decreasing of the functional $F$. The proof is omitted since it is analogous to that of Proposition~\ref{prop:lastb}.

\begin{proposition}\label{prop:lasts}
Let $m_o > 0$ satisfy \eqref{eq:cximos}. Moreover, assume that $\xi$, $K_{\eta,\zeta}^{\ell,m,r}$ and $\rho$ satisfy \eqref{eq:cxis}--\eqref{eq:klrezs} and \eqref{eq:composite7}. Then, the following two statements hold.
\begin{itemize}
\item \emph{Local Decreasing}. For any interaction at time $t>0$ between two waves satisfying \eqref{rogna}, it holds
\begin{equation*}
\Delta F(t) \le 0\,.
\end{equation*}
\item \emph{Global Decreasing}. If
\begin{equation}
	\label{eq:boundLds}
	\bar{L}(0)\le m_o\hspace{0.8pt} c(m_o)
\end{equation}
and the approximate solution is defined in $[0,T]$, then $F(0)\le m_o$, $\Delta F(t)\le 0$ for every $t\in(0,T]$ and \eqref{rogna} is satisfied.
\end{itemize}
\end{proposition}


\section{End of the Proof of Theorem~\ref{thm:main}}\label{sec:convergence}
\setcounter{equation}{0}

In this last section, we conclude the proof of Theorem~\ref{thm:main} and add some final comments. The proof fits in the general framework detailed in \cite{ABCD2}; hence, we only outline the most important changes.

First, as in \cite[Section 6]{ABCD2} the front tracking algorithm used to construct the approximate solutions is well-defined and converges. Moreover, it is consistent in the sense that the total size of the non-physical waves carried by the composite waves vanishes with $\nu$. In brief, to estimate this quantity we use the notion of generation order, i.e.\ we attach an index $k\ge 1$ to each wave generated in the construction. Then, according to each $k$, we introduce suitable functionals $L_k,Q_k,F_k$ simply by referring the functionals $L,Q,F$ to waves with order $k$. In particular, following the same steps as in \cite{ABCD,ABCD2}, in both cases we can prove that
\begin{equation}\label{eq:decFk}
\tilde{F}_k(t)=\sum_{j\ge k} F_j(t) \le {\mu^{k-1}} F_1(0),
\end{equation}
where $\mu \in\,]0,1[$ is either $\mu_b$ or $\mu_c$. More precisely, in the bubble case we find
\begin{equation*}
\begin{aligned}
  \mu_b  = \max\bigg\{&\frac{1+K_\zeta^m|\zeta|}{2K_\eta^m-1}\,,\,\frac{1+K_\eta^m|\eta|}{2K_\zeta^m-1}\,,\,\frac{\xi}{1+2K_\eta^{\ell}}\,,\,\frac{\xi}{1+2K_\zeta^{r}}\,,\,\frac{1+K_\eta^m|\eta|}{\xi}\,,\,\frac{1+K_\zeta^m|\zeta|}{\xi}\,,\,\\
  &\frac{1+K_\eta^\ell|\eta|+K_\zeta^\ell|\zeta|}{\xi}\,,\,\frac{1+K_\eta^r|\eta|+K_\zeta^r|\zeta|}{\xi}\,,\,\frac{C_o}{\xi(2K_\eta^m-C_o)}\,,\,\frac{C_o}{\xi(2K_\zeta^m-C_o)}\bigg\}\,;
\end{aligned}
\end{equation*}
while in the increasing-pressure case we have
\begin{equation*}
\begin{aligned}
  \mu_c  = \max\bigg\{&\frac{1+K_\zeta^m|\zeta|}{2K_\eta^m-1}\,,\,\frac{\xi(1+K_\eta^m|\eta|)}{1+2K_\zeta^m}\,,\,\frac{1+K_\eta^m|\eta|}{\xi}\,,\,\frac{1+K_\eta^\ell(|\eta|+|\zeta|)}{\xi}\,,\,\frac{1+K_\eta^r(|\eta|+|\zeta|)}{\xi}\,,\,\\
  &\frac{\xi|\eta|/2}{(K_\eta^\ell-1/2)|\eta|+(K_\zeta^\ell-K_\zeta^m)|\zeta|}\,,\,\frac{|\zeta|/2}{(K_\zeta^r-1/2)|\zeta|+[K_\eta^r-K_\eta^m(1+|\zeta|/2)]|\eta|}\,,\\
  &\frac{C_o|\zeta|/2}{\xi(K_\zeta^r-C_o/2)|\zeta|+\xi[K_\eta^r-K_\eta^m(1+C_o|\zeta|/2)]|\eta|}\,,\,\frac{C_o}{\xi(2K_\eta^m-C_o)}\bigg\}\,.
\end{aligned}
\end{equation*}
In both cases, simple calculations show that $\mu<1$ because of our choice of keeping strict inequalities in the final parts of Sections~\ref{sec:interbubble} and~\ref{sec:interdstep}. We exploit formula \eqref{eq:decFk} to show that the total size of the composite waves tends to zero, as follows. We have
\begin{align}
    \lefteqn{\left[\text{total size of composite waves}\right]\le}\nonumber \\
    &\le\,\left[\text{size of composite waves of order $\ge k$}\right]+\, \left[\text{size of composite waves of order $<k$}\right] \nonumber\\
    &\le  \, \mu^{k-1}\cdot F_1(0)+\frac{\rho}{2}\,C_o(\rho)(|\eta|+|\zeta|)\,
    \left[\text{number of fronts of order $<k$}\right] \nonumber\\
    &\le \, \mu^{k-1}\cdot m_o+ \frac{\rho}{2}\,C_o(\rho)(|\eta|+|\zeta|)\,
    \left[\text{number of fronts of order $<k$}\right]\,.\label{eq:finale}
\end{align}
Then, \eqref{eq:finale} is less than $1/\nu$ if we choose $k$ sufficiently large to have the first term less than $1/(2\nu)$ and $\rho=\rho_{\nu}(m_o)$ small enough to have the second term less than $1/(2\nu)$.


\begin{proof}[End of the proof of Theorem~\ref{thm:main}] As in the proof of \cite[Theorem 2.1]{ABCD2}, in the bubble case (increasing-pressure case) by \eqref{eq:hyp2} and \eqref{eq:boundLb} (\eqref{eq:boundLds}, respectively) we prove that
\begin{equation}\label{eq:tvpm}
\bar{L}(0)\le\frac{1}{2}\tv\left(\log(p_o)\right)+ \frac{1}{2\min\{a_\ell,a_m,a_r\}}\tv\left(u_o\right)\,.
\end{equation}
Now, by \eqref{eq:ximo1} (\eqref{eq:cximos}, respectively) and \eqref{eq:tvpm} we look for an $m_o$ satisfying both inequalities below:
\begin{align}
    \mathcal{H}(|\eta|,|\zeta|)&<\frac{1}{{c(m_o)}}-1 \,=\, \frac{2}{\cosh m_o -1} \, =: \, w(m_o)\,, \label{hyp1}\\
\tv\left(\log(p_o)\right)+ \frac{1}{\min\{a_\ell,a_m,a_r\}}\tv\left(u_o\right)
& < 2m_o\hspace{.8pt} c(m_o) \, =:\, z(m_o)\,. \label{hyp2-1}
\end{align}
Recall from \cite{ABCD2} that $w(r)$ is strictly decreasing and $z(r)$ is strictly increasing, for $r\in\R^+$. Moreover, we have $\mathcal{K}(r)=z\left(w^{-1}(r)\right)$, see \eqref{eq:defK}.
Hence, by \eqref{eq:hyp2} 
one can choose $m_o$ such that \eqref{hyp1}, \eqref{hyp2-1} hold in both the cases. Therefore, we can conclude as in \cite{Bressanbook} and Theorem~\ref{thm:main} is completely proved.
\end{proof}

\bigskip

As in \cite{ABCD2}, we want to compare the results obtained here in the bubble case
with that of \cite{amadori-corli-siam}. More precisely, we set $x=|\eta|$, $y=|\zeta|$ and we claim that
\begin{equation}
    \label{eq:bubsima}
    \mathcal{H}_b(x,y)\le x+y\qquad\text{for }0\le x+y<1/2.
\end{equation}
Since $\mathcal{H}_b$ is a symmetric function of $x$ and $y$, it suffices to verify that in the common domain it holds
\begin{equation}
    \label{eq:bubsima1}
    \frac{(2+x)4y}{(2-x)(4-xy)}<x+y\,.
\end{equation}
By simplifying expression \eqref{eq:bubsima1}, we find that it is equivalent to
\[
x^2y + xy^2 -(2xy + 2y^2 + 4x + 8y) + 8 >0\,,
\]
which will be satisfied if $xy + y^2 + 2x + 4y < 4$. Since $x<1/2-y$, this last inequality is verified if
\[
\left(\frac12-y\right)y + y^2 + 2\left(\frac12-y\right) + 4y < 4\,,
\]
that is when $y < 6/5$. Therefore, \eqref{eq:bubsima} holds and, since $\mathcal{K}$ is decreasing, we have
$$
\mathcal{K}\left(\mathcal{H}_{b}(|\eta|,|\zeta|)\right) >\mathcal{K}\left(|\eta| + |\zeta|\right)
$$
in the common domain $|\eta|+|\zeta|<1/2$. Hence, Theorem~\ref{thm:main} improves \cite[Theorem 2.2]{amadori-corli-siam} in the bubble case.


{\small
\bibliographystyle{abbrv}

}



\end{document}